\documentclass[aos,preprint]{imsart}

\RequirePackage[OT1]{fontenc}
\RequirePackage{amsthm,amsmath,amsfonts}
\RequirePackage[numbers]{natbib}
\RequirePackage[colorlinks,citecolor=blue,urlcolor=blue]{hyperref}
\RequirePackage{hypernat}
\RequirePackage[pdftex]{graphics}
\RequirePackage{graphicx}
\RequirePackage{float}

\startlocaldefs
\numberwithin{equation}{section}
\theoremstyle{plain}
\newtheorem{theorem}{Theorem}[section]
\newtheorem{lemma}{Lemma}[section]
\newtheorem{proposition}{Proposition}[section]

\newtheorem{example}{Example}[section]

\startlocaldefs
\endlocaldefs

\begin{document}

\begin{frontmatter}

\title{CONSISTENCY of SPARSE PCA IN HIGH DIMENSION, LOW SAMPLE SIZE CONTEXTS}
\runtitle{CONSISTENCY of SPARSE PCA IN HDLSS}


\begin{aug}
\author{\fnms{Dan} \snm{Shen}\thanksref{m1,t1}\ead[label=e1]{dshen@email.unc.edu}},
\author{\fnms{Haipeng} \snm{Shen}\thanksref{t2}\ead[label=e2]{haipeng@email.unc.edu}}
\and
\author{\fnms{J. S.} \snm{Marron}\thanksref{t3}
\ead[label=e3]{marron@email.unc.edu}}

\thankstext{m1}{Corresponding Author}
\thankstext{t1}{Partially supported by NSF grant DMS-0854908}
\thankstext{t2}{Partially supported by NSF grants DMS-0606577 and CMMI-0800575}
\thankstext{t3}{Partially supported by NSF grants DMS-0606577 and DMS-0854908}
\runauthor{Dan Shen, Haipeng Shen and J. S. Marron}

\affiliation{University of North Carolina at Chapel Hill}

\address{Department of Statistics and Operations Research\\
University of North Carolina at Chapel Hill\\
Chapel Hill, NC 27599\\
\printead{e1}\\
\phantom{E-mail:\ }\printead*{e2}\\
\phantom{E-mail:\ }\printead*{e3}
}
\end{aug}

\begin{abstract}
Sparse Principal Component Analysis (PCA) methods are efficient tools to reduce the dimension (or the number of variables)
of complex data.  Sparse principal components (PCs) are easier to interpret than conventional PCs, because most loadings are zero.  We study the asymptotic properties of these sparse PC directions for scenarios with fixed sample size and increasing dimension (i.e. High Dimension,
Low Sample Size (HDLSS)). Under the previously studied spike covariance assumption, we show that Sparse PCA remains consistent under the same large spike condition that was previously established for conventional PCA. Under a broad range of small spike conditions, we find a large set of sparsity
assumptions where Sparse PCA is consistent, but PCA is strongly inconsistent. The boundaries of the consistent region are clarified using an oracle result.
\end{abstract}



\end{frontmatter}


\section{Introduction}\label{sec:01}

Principal Component Analysis (PCA) is  an important visualization and dimension
 reduction tool for High Dimension, Low Sample Size (HDLSS) data.
However, the linear combinations found by PCA typically will involve all the variables, with non-zero loadings, which can be challenging to interpret.
In order to overcome this weakness of PCA, we will study sparse PCA methods that generate sparse principal components (PCs), i.e. PCs with only a few non-zero loadings. Several sparse PCA methods have been proposed to facilitate the interpretation of HDLSS data, see for example Zou, Hastie and Tibshirani (2006)~\cite{zou2006sparse},
Shen and Huang (2008)~\cite{shen2008sparse}, Leng and Wang (2009)~\cite{leng2009general}, Witten, Tibshirani and Hastie (2009)~\cite{witten2009penalized},
Johnstone and Lu (2009)~\cite{johnstone2009consistency},
Amini and Wainwright (2009)~\cite{amini2009high}, and Ma (2010)~\cite{ma2010}.

This paper studies the HDLSS asymptotic properties of sparse PCA. HDLSS asymptotics are based on the limit, as the dimension $d\rightarrow\infty$, with the sample size $n$ fixed, as originally studied by Hall, Marron and Neeman (2005)~\cite{hall2005geometric} and Ahn et al. (2007)~\cite{ahn2007high}.
Theoretical properties of sparse PCA have been studied before under different asymptotic frameworks. Leng and Wang (2009)~\cite{leng2009general} used the {\it adaptive lasso} penalty of Zou, Hastie and Tibshirani (2006)~\citep{zou2006adaptive} to introduce sparse loadings, and established some consistency result for selecting non-zero loadings when the sample size $n\rightarrow \infty$, with the dimension $d$ fixed. Johnstone and Lu (2009)~\cite{johnstone2009consistency} considered a single-component spiked covariance model (originally proposed by Johnstone (2001)~\cite{johnstone2001distribution}) and showed that conventional PCA is consistent if and only if $d(n)/n\rightarrow 0$; furthermore, under the condition $\log(d\vee n)/n\rightarrow 0$, they proved consistency of PCA performed on a subset of variables with largest sample variance.  Amini and Wainwright (2009)~\cite{amini2009high} considered the same single-component spiked model, and further restricted the maximal eigenvector to have $k$ non-zero entries; they studied the thresholding subset PCA procedure of Johnstone and Lu~\cite{johnstone2009consistency} and the sparsePCA procedure of d'Aspremont et al. (2007)~\cite{d2007direct}, and explored conditions on the triplet $(n, d, k)$ under which each procedure can recover the support set of the sparse eigenvector with probability one. Paul and Johnstone~\cite{paul2007augmented} developed the augmented sparse PCA procedure along with its optimal rate of convergence property. Ma~\cite{ma2010} proposed an iterative thresholding procedure for estimating principal subspaces that has nice theoretical properties.

Sparse PCA is primarily motivated by modern data sets of very high dimension; hence we prefer the statistical viewpoint of the High Dimension Low Sample Size (HDLSS) asymptotics. Note that this case of $d\rightarrow \infty$ with $n$ fixed was not considered by Johnstone and Lu~\cite{johnstone2009consistency}. Conventional PCA was first studied using HDLSS asymptotics by Ahn et al.~\cite{ahn2007high} and the most comprehensive current
result is Jung and Marron (2009)~\cite{jung2009pca}. The latter found conditions when the first several empirical PC directions
 would be \emph{consistent} or \emph{subspace consistent}
with the corresponding population PC directions. This happens when the first several eigenvalues are large enough, compared with
the rest of the
eigenvalues of the population covariance matrix. Moreover, if the first few
eigenvalues are not sufficiently large, all empirical
PC directions will be \emph{strongly inconsistent} with their population counterparts in the sense
that the angle between them will converge to 90 degrees.

The main contribution of this paper is an exploration of conditions where  conventional PCA is strongly inconsistent
(for scenarios with relatively small population eigenvalues), yet sparse PCA methods are consistent.
Furthermore, the mathematical boundaries of the sparse PCA consistency are established by showing strong
inconsistency, for even an oracle version of sparse PCA, beyond the consistent region. Similar to~Johnstone and Lu (2009)~\cite{johnstone2009consistency} and
Amini and Wainwright (2009)~\cite{amini2009high}, we focus on the single component spiked covariance model. Our results depend on a \emph{spike index}, $\alpha$, defined below in the context of Example~\ref{example:01}, which measures
the dominance of the first eigenvalue, and on a \emph{sparsity index}, $\beta$, defined also in Example~\ref{example:01},
which measures the number of non-zero entries of the first population eigenvector. For illustration purposes, we simplify the consistency and strong inconsistency results for the exemplary model considered in Example~\ref{example:01}, and summarize them below as functions of $\alpha$ and $\beta$ in Figure~\ref{fig:01}:

\begin{itemize}
\item \textbf{Previous Results (dark grey rectangle)}: Jung and Marron (2009) \cite{jung2009pca} showed that the first empirical eigenvector
 is consistent with
the first population eigenvector when the spike index $\alpha$ is greater than 1.
\item \textbf{Consistency (white triangle)}: We will show that sparse PCA is consistent
 even when the spike index $\alpha$ is less than 1, as long as $\alpha$ is greater than the sparsity index $\beta$.
 This is done in Section~\ref{sec:02} for a simple thresholding method and in Section~\ref{sec:03} for the RSPCA method proposed by Shen and Huang (2008)~\cite{shen2008sparse}.
\item \textbf{Strong Inconsistency (black triangle)}: In Section~\ref{sec:04} we show that even an oracle sparse PCA procedure is
strongly inconsistent with the first population eigenvector, when the spike index $\alpha$ is smaller than the sparsity index $\beta$.
\item \textbf{Irrelevant Area (light grey rectangle)}: The sparsity index $\beta$ can not be larger than 1, hence the light grey rectangular
area is irrelevant.
\end{itemize}

\begin{figure}[h]
\vspace{-.5cm}
 \begin{center}
 \includegraphics[width=\textwidth ]{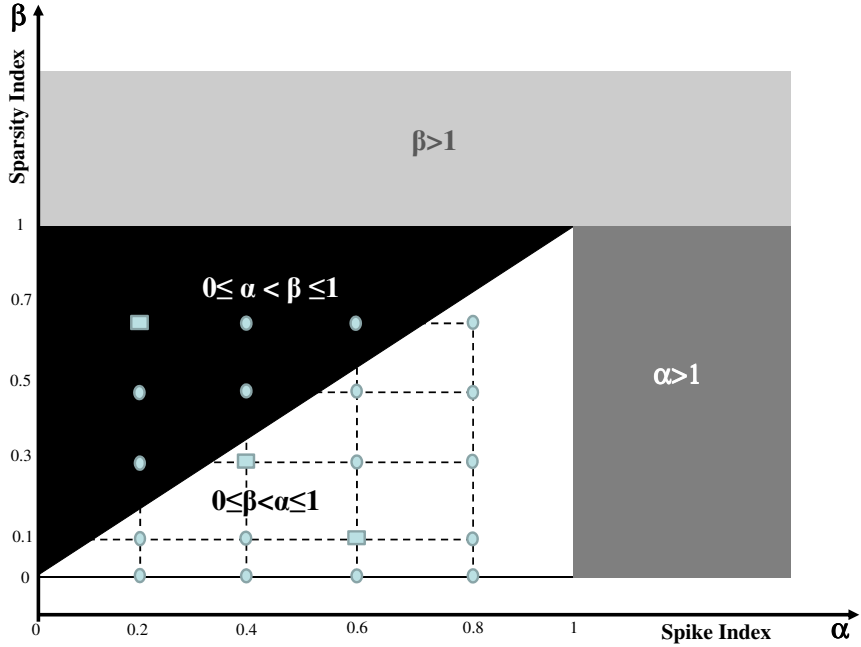}\vspace{-.3cm}
 \end{center}
 \vspace{-.4cm}
 \caption{Consistent areas for PCA and sparse PCA, as a function of the spike index $\alpha$ and
 the sparsity index $\beta$, under the single component spiked model considered in Example~\ref{example:01}. Conventional PCA is consistent only on
 the dark grey rectangle $(\alpha>1)$, while sparse PCA is also  consistent on the white triangle $(0\leq\beta<\alpha\leq1)$.
 In addition, an oracle sparse PCA procedure
 is strongly inconsistent on the black triangle $(0\leq\alpha<\beta\leq1)$. The light grey rectangular area $(0\leq \alpha<1, \beta>1)$
 is not considered because the sparsity index $\beta\leq1$. The dots show the grid points studied in the simulation study of Section~\ref{sec:04}.
 }\label{fig:01}
 \vspace{-.3cm}
 \end{figure}

\subsection{Notation and Assumptions}\label{subsec:11}

All quantities are indexed by the dimension $d$ in this paper. However, when it will not lead to confusion,
the subscript $d$ will be omitted for convenience. Let
the population covariance matrix be $\Sigma_d$.  The eigen-decomposition of $\Sigma_d$ is
\begin{equation*}
\Sigma_d=U_d\Lambda_dU^T_d,
\end{equation*}
where $\Lambda_d$ is the diagonal matrix of the population eigenvalues
$\lambda_1\geq\lambda_2\geq\ldots\geq\lambda_d$ and $U_d$
is the  matrix of corresponding population eigenvectors so that
$U_d=[u_1,\cdot\cdot\cdot,u_d]$.

Assume that $X_1,\ldots, X_n$ are random samples from a
$d$-dimensional normal distribution $N(0, \Sigma_d)$. Denote the
data matrix by $X_{(d)}=[X_1,\ldots,X_n]_{d\times n}$ and the sample
covariance matrix by $\hat{\Sigma}_d=n^{-1}X_{(d)}X_{(d)}^{T}$.
Then, the sample covariance matrix $\hat{\Sigma}_d$ can be similarly
decomposed as
\begin{equation*}
\hat{\Sigma}_d=\hat{U}_d\hat{\Lambda}_d\hat{U}^T_d,
\end{equation*}
where $\hat{\Lambda}_d$ is the diagonal matrix of the sample eigenvalues
$\hat{\lambda}_1\geq\hat{\lambda}_2\geq\ldots\geq\hat{\lambda}_d$
and $\hat{U}_d$ is the matrix of the corresponding sample eigenvectors so that
$\hat{U}_d=[\hat{u}_1,\ldots,\hat{u}_d]$.

Let $\bar{u}_i$ be any sample based estimator of $u_i$, e.g. $\bar{u}_i=\hat{u}_i$ for $i=1,\ldots ,d$. Two important concepts from
Jung and Marron (2009)~\cite{jung2009pca} are:
\begin{itemize}
\item \textbf{Consistency}: The direction $\bar{u}_i$ is \emph{consistent} with
its population counterpart $u_i$ if
\begin{equation}\mbox{Angle}(\bar{u}_i,u_i)\equiv\mbox{arccos}(\mid<\bar{u}_i,u_i>\mid)\xrightarrow{p}
0, \mbox{as} \;d\rightarrow \infty\label{Consistency}
,\end{equation}
where $ <\cdot, \cdot> $ denotes the inner product
between two vectors.
\item \textbf{Strong Inconsistency}: The direction $\bar{u}_i$ is \emph{strongly inconsistent} with
its population counterpart $u_i$ if
\begin{equation*}\mbox{Angle}(\bar{u}_i,u_i)=\mbox{arccos}(\mid<\bar{u}_i,u_i>\mid)\xrightarrow{p}
\frac{\pi}{2}, \mbox{as} \;d\rightarrow \infty.
\label{eq:inconsistency}
\end{equation*}
\end{itemize}

In addition, we consider another important concept in the current paper:
\begin{itemize}
\item
\textbf{Consistency with convergence rate $d^\iota$}: The direction $\bar{u}_i$ is consistent with
its population counterpart $u_i$ with the convergence rate $d^\iota$ if
$\mid<\bar{u}_i,u_i>\mid=1+o_p(d^{-\iota})$, where  the notation  $G_d\equiv o_p(d^{-\iota})$ means that $d^{\iota}G_d\xrightarrow{p} 0$,
as $d\rightarrow \infty$.
\end{itemize}

\begin{example} Assume that $X_1,\ldots, X_n$ are random sample vectors
from a $d$-dimensional normal distribution $N(0, \Sigma_d)$, where the covariance matrix $\Sigma_d$
has the eigenvalues as
\begin{equation*}
\lambda_1=d^\alpha, \lambda_2=\ldots=\lambda_d=1, \alpha\geq0.\label{eigenvalues}
\end{equation*}
This is a special case of the single component spike covariance Gaussian model considered before by, for example,
Johnstone (2001)~\cite{johnstone2001distribution}, Paul (2007) \cite{paul2007asymptotics},
Johnstone and Lu (2009)~\cite{johnstone2009consistency},
Amini and Wainwright (2009)~\cite{amini2009high}.
Without loss of generality (WLOG), we further assume that the first eigenvector $u_1$ is proportional to
the following $d$-dimensional vector
\begin{equation*}
\dot{u}_1=( \overbrace{1,\ldots,1}^{\lfloor d^{\beta}\rfloor}, 0,\ldots,0)^T,
\label{first_eigenvector}
\end{equation*}
where $0\leq \beta \leq1$ and $\lfloor d^{\beta}\rfloor$ denotes the integer part of $d^\beta$. (In general the non-zero entries do not have to be the first $\lfloor d^{\beta}\rfloor$ elements, neither do they need to be equal.) If $\beta=0$, the first population eigenvector becomes $u_1=(1,0,\ldots,0)^T$.
\label{example:01}
\end{example}

For the above model, Jung and Marron (2009)~\cite{jung2009pca} showed that
the first empirical eigenvector (the PC direction) $\hat{u}_1$ is
consistent with $u_1$ when $\alpha>1$; however for $\alpha<1$, it is
strongly inconsistent. Again, the main point of the current paper is
an exploration of conditions under which sparse methods can lead to
consistency when the spike index $\alpha \leq 1$, (recall that the
first eigenvalue $\lambda_1=d^\alpha$), by exploiting
\emph{sparsity}. Sparsity is quantified by the sparsity index
$\beta$, where $\lfloor d^\beta \rfloor$ is the number of non-zero
elements of the first eigenvector $u_1$. Here we use the above
simple example for intuitive illustration purposes, to highlight the
key findings. More general single component spike models will be
considered in Sections~\ref{sec:02} to~\ref{sec:04}.

\subsection{Roadmap of the paper} \label{subsec:12}
The organization of the rest of paper is as follows. For easy access to the main ideas, Section~\ref{sec:02} first introduces a simple thresholding method to generate sparse PC directions. Section~\ref{subsec:21} shows the consistency of the sparse PC directions, obtained by this simple thresholding method. Section~\ref{sec:03} then generalizes these ideas to a current sparse PCA method. In particular, we consider the sparse PCA method developed by Shen and Huang (2008)~\cite{shen2008sparse}, and build its connection to the simple thresholding method. 
We then establish the consistency
of the sparse PCA method under the sparsity and small spike conditions where the conventional PCA is strongly inconsistent. Section~\ref{sec:04} considers scenarios when the spike index $\alpha$ is smaller than the sparsity index $\beta$, and proves the strong inconsistency of an appropriate oracle PCA procedure. Section~\ref{sec:05} reports some simulation results to illustrate both consistency and strong inconsistency of PCA and sparse PCA. Section~\ref{sec:06} concludes the paper with some discussion of future work on extending consistency of sparse PCA to more general distributions. We point out that it is challenging to move beyond Gaussianity to get HDLSS consistency of sparse PCA. Section~\ref{sec:07} contains the proofs of the theorems.

\section{Consistency of a simple thresholding method for sparse PCA  in HDLSS}\label{sec:02}
In Example~\ref{example:01}, the first eigenvector of the sample
covariance matrix $\hat{u}_1$ is strongly inconsistent with $u_1$
when $\alpha<1$, because it attempts to estimate too many
parameters. Sparse data analytic methods assume that many of these
parameters are zero, which can allow greatly improved estimation of
the first PC direction $u_1$. Here, this issue is explored in the
context of sparse PCA, where $u_1=(1, 0, \ldots,0)^T$ is an extreme
case. The sample covariance matrix based estimator, $\hat{u}_1$, can
be improved by exploiting the fact that $u_1$ has many zero
elements.

A natural approach is a simple thresholding method where entries
with small absolute values are replaced by zero. In HDLSS contexts,
it is challenging to apply thresholding directly to the entries of
$\hat{u}_1$, because the number of them grows rapidly as
$d\rightarrow\infty$, which naturally shrinks their magnitudes given
that $\hat{u}_1$ is a unit vector. Thresholding is more conveniently
formulated in terms of the \emph{dual covariance matrix} as used by
Jung, Sen and Marron (2010)~\cite{Jung2010}.

Denote the dual sample covariance matrix by $S_d= \frac{1}{n}
X_{(d)}^{T}X_{(d)}$ and the first dual eigenvector by $\tilde{v}_1$.
The sample eigenvector $\hat{u}_1$ is connected with the dual
eigenvector $\tilde{v}_1$ through the following transformation,
\begin{equation}
\tilde{u}_1=(\tilde{u}_{1,1},\ldots,\tilde{u}_{d,1})^T=X_{(d)}\tilde{v}_{1},
\label{eq:dualeigevector}
\end{equation}
and the sample estimate is then given by
$\hat{u}_1={\tilde{u}_1}/{\|\tilde{u}_1\|}$~\cite{Jung2010}.

Given a sequence of threshold values $\lambda$, define
the thresholded entries as
\begin{equation} \breve{u}_{i,1}=
\begin{cases} \tilde{u}_{i,1} & \text{if $|\tilde{u}_{i,1}|> \lambda $,}
\\
0 &\text{if $|\tilde{u}_{i,1}|\leq \lambda $,}
\end{cases}
\quad {\rm for}\quad i=1,\ldots,d.\label{simplethreshold}
\end{equation}
Denote $\breve{u}_1=(\breve{u}_{1,1},\ldots,\breve{u}_{1,1})^T$ and
normalize it to get the simple thresholding (ST) estimator
$\hat{u}^{\rm ST}_1={\breve{u}_1}/{\|\breve{u}_1\|}$.

For the model considered in Example~\ref{example:01}, given an
eigenvalue of strength $\alpha \in (0,1)$, (recall
$\lambda_1=d^\alpha$ and $\hat{u}_1$ is strongly inconsistent),
below we explore conditions on the threshold sequence
$\lambda$ under which the ST estimator $\hat{u}^{\rm ST}_1$ is in
fact consistent with $u_1$. First of all, the threshold $\lambda$
can not be too large; otherwise all the entries will be zeroed out.
It will be seen in Theorem~\ref{Th:01} that a sufficient condition
for this is $\lambda\leq d^{\frac{\gamma}{2}}$, where $\gamma\in(0,
\alpha)$. Secondly, the threshold $\lambda$ can not be too small, or
pure noise terms will be included. A parallel sufficient condition
is shown to be $\lambda\geq \log^\delta(d)$, where
$\delta\in(\frac{1}{2},\infty)$.

\subsection{Consistency of the simple thresholding method}\label{subsec:21}
Below we formally establish conditions on the eigenvalues of the population covariance matrix $\Sigma_d$ and the thresholding parameter $\lambda$, which give consistency of $\hat{u}_1^{\mbox{ST}}$ to $u_1$. All the technical proofs are provided in Section~\ref{sec:07} and the supplement materials.

We begin with considering the extreme sparsity case
$u_1=(1,0,\ldots,0)^T$. Suppose that $\lambda_1\sim d^\alpha$, in
the sense that $0<c_1\leq
\underline{\mbox{lim}}_{d\rightarrow\infty}\frac{\lambda_1}{d^\alpha}\leq
 \overline{\mbox{lim}}_{d\rightarrow\infty}\frac{\lambda_1}{d^\alpha} \leq c_2 $, where $c_1$ and
 $c_2$ are two constants.  Similarly, assume $\sum_{i=2}^d \lambda_i\sim d$. As in Jung and Marron~\cite{jung2009pca}, denote the measure of sphericity as
 \begin{eqnarray*}
 \varepsilon\equiv \frac{\mbox{tr}^2(\Sigma_d)}{d \mbox{tr}(\Sigma_d^2)}=\frac{(\sum_{i=1}^d \lambda_i)^2}{d\sum_{i=1}^d\lambda_i^2 },
 \end{eqnarray*}
 and assume the  $\varepsilon$-condition: $\varepsilon\gg \frac{1}{d}$, i.e
 \begin{eqnarray}
 (d\varepsilon)^{-1}= \frac{\sum_{i=1}^d \lambda_i^2}{(\sum_{i=1}^d\lambda_i)^2 }
 \rightarrow 0, \mbox{as}\; d\rightarrow\infty.\label{eqn:eps}
 \end{eqnarray}

Now we need to impose the following conditions on the eigenvalues:
\begin{itemize}
\item Assume that $\overline{\mbox{lim}}_{d\rightarrow\infty}\frac{\lambda_1}{\sum_{i=2}^d \lambda_i}= c$, where $c$ is a non-negative
        constant, and the $\varepsilon$-condition is satisfied. These conditions can  guarantee that the dual matrix
         $S_d$ has a limit. Hence the first dual eigenvector $\hat{v}_1$ will have a limit and it will then help build up the
          consistency of $\hat{u}^{\rm ST}_1$.
\item
        In addition, we need the second eigenvalue $\lambda_2$ to be an obvious
         distance away from the first eigenvalue $\lambda_1$. If not, it will be hard to
         distinguish the first and second empirical eigenvectors as observed by Jung and Marron, among others.
        In that case the appropriate amount of thresholding on the first empirical eigenvector becomes unclear. Therefore, we assume
          that $\lambda_2\sim d^\theta$, where $\theta<\alpha$.
\end{itemize}

\begin{theorem} Suppose that $X_1,\ldots, X_n$ are
random samples from a $d$-dimensional normal distribution $N(0,
\Sigma_d)$ and the first population eigenvector $u_1=(1,0,\ldots,0)^T$.
If the following conditions are satisfied:
\begin{enumerate}
\item[(a)]$\lambda_1\sim d^\alpha$, $\lambda_2\sim d^\theta$,  and $\sum_{i=2}^d\lambda_i\sim d$,
where $\theta\in[0,\alpha)$ and $\alpha\in(0,1]$,
\item[(b)]for a non-negative constant $c$, $\overline{\mbox{lim}}_{d\rightarrow\infty}\frac{\lambda_1}{\sum_{i=2}^d \lambda_i}= c$
and the $\varepsilon$-condition~\eqref{eqn:eps} is satisfied,
\item[(c)]$\log^\delta(d)\; d^{\frac{\theta}{2}}\leq \lambda
 \leq d^{\frac{\gamma}{2}}$, where $\delta\in(\frac{1}{2},\infty)$ and
$\gamma\in(\theta, \alpha)$,
\end{enumerate}
then the simple thresholding estimator $\hat{u}^{\rm ST}_1$  is consistent with $u_1$.\label{Th:01}
\end{theorem}

In fact, $u_1=(1,0,\ldots,0)^T$ in Theorem~\ref{Th:01} is a very extreme case. The following theorem considers
the general case $u_1=(u_{1,1},\ldots,u_{d,1})^T$, where
only $\lfloor d^\beta \rfloor$ elements of $u_1$ are non-zero.
WLOG, we assume that the first $ \lfloor d^\beta \rfloor$ entries are non-zero just for notational convenience.

Define \begin{equation}
Z_j\equiv(z_{1,j},\ldots,z_{d,j})^T=(X^T_j u_1,\ldots,X^T_j u_d)^T,\quad j=1,\ldots,n.\label{eq:z}
\end{equation}
We can show that $Z_j$ are iid  $N\left(0, {\rm diag}\{\lambda_1,\ldots,\lambda_d\}\right)$ random vectors. In addition, let
\begin{equation}
W_j\equiv(w_{1,j},\ldots, w_{d,j})^T=(\lambda^{-\frac{1}{2}}_1z_{1,j},
\ldots,\lambda^{-\frac{1}{2}}_dz_{d,j})^T,\quad j=1,\ldots,n, \label{eq:w}
\end{equation}
and the $W_j$ are iid $N(0, I_d)$ random vectors, where $I_d$ is the $d$-dimensional identity matrix.

The following additional conditions are needed to ensure the
consistency of $\hat{u}^{\rm ST}_1$: 
\begin{itemize}
\item The non-zero entries of the population eigenvector $u_1$ need to be a certain distance away from zero. In fact, if  the non-zero entries of the first population eigenvector are close to zero,
the corresponding entries of the first empirical eigenvector would also be
 small and look like pure noise entries. Thus, we assume
\begin{eqnarray*}
\mbox{max}_{1\leq i \leq \lfloor d^\beta \rfloor}
|u_{i,1}|^{-1}\sim d^{\frac{\eta}{2}},\quad {\rm where}\quad
\eta\in[0,\alpha).
\end{eqnarray*}

\item  From (\ref{eq:z}), we have
\begin{eqnarray*}
X_j=\sum_{i=1}^d z_{i,j}u_i,\quad j=1,\ldots,n.
\end{eqnarray*}
Since $z_{1,j}$ has the largest variance $\lambda_1$, then $z_{1,j}u_1$ contributes the most to
the variance of $X_j$, $j=1,\ldots,n$. Note that $z_{1,j}u_1$ is consistent with $u_1$,
and so $z_{1,j}u_1$ is the key to making the simple thresholding method work. So we
need to show that the remaining parts
\begin{eqnarray}
H_j\equiv(h_{1,j},\ldots,h_{d,j})^T=\sum_{i=2}^d z_{i,j}u_i, \quad j=1,\ldots,n\label{eq:h}
\end{eqnarray}
have a negligible effect on the direction vector $\hat{u}^{\rm ST}_1$.

\item Suppose that
the $H_j$ are iid $N(0,\Delta_d)$, where $\Delta_d=(m_{k l})_{d\times d}$, for $j=1,\ldots,n$.
A sufficient condition to make their effect negligible is the following mixing condition
of Leadbetter, Lindgren and Rootzen (1983)~\cite{leadbetter1983extremes}:
\begin{eqnarray}|m_{kl}|\leq
{m_{kk}}^{\frac{1}{2}}{m_{ll}}^{\frac{1}{2}} \rho_{|k-l|}, \quad
1\leq k\neq l \leq \lfloor d^\beta \rfloor,\label{mixing}
\end{eqnarray}
where $\rho_t < 1$ for all $t>1$ and $\rho_t \log(
t)\longrightarrow0$,  as $ t \rightarrow\infty$. This mixing
condition can guarantee that $\mbox{max}_{1\leq j\leq n}|h_{1,j}|$
has a quick convergence rate, as $d\rightarrow \infty$. It enables
us to neglect the influence of $H_j$ for sufficiently large $d$ and
make $z_{i,j}u_1$ the dominant component, which then gives
consistency to the first population eigenvector $u_1$. Thus the
thresholding estimator $\hat{u}_1^{\rm ST}$ becomes consistent.
\end{itemize}

We now state one of the main theorems:
\begin{theorem}  Assume that
$X_1,\ldots,X_n$ are random samples from a
$d$-dimensional normal distribution $ N(0, \Sigma_d) $. Define $Z_j$, $W_j$ and $H_j$ as in (\ref{eq:z}), (\ref{eq:w}), and (\ref{eq:h}) for $j=1,\ldots,n$.
The first population eigenvector is $u_1=(u_{1,1},\ldots,u_{d,1})^T$ with $u_{i,1}\neq 0, i=1,\ldots, \lfloor d^\beta \rfloor$, and otherwise
$u_{i,1}=0$.

If the following conditions are satisfied:
\begin{enumerate}
\item[(a)] $\lambda_1\sim d^\alpha$, $\lambda_2\sim d^\theta$,  and $\sum_{i=2}^d\lambda_i\sim d$,
where $\theta\in[0,\alpha)$ and $\alpha\in(0,1]$,

\item[(b)] for a non-negative constant $C$, $\overline{\mbox{lim}}_{d\rightarrow\infty}\frac{\lambda_1}{\sum_{i=2}^d \lambda_i}=C$
and  $\varepsilon$-condition \eqref{eqn:eps} is satisfied,

\item[(c)] $\mbox{max}_{1\leq i \leq [d^\beta]}
|u_{i,1}|^{-1}\sim d^{\frac{\eta}{2}}$, where $\eta\in[0,\alpha)$,

\item[(d)] $H_j$ satisfies the mixing condition (\ref{mixing}), $j=1,\ldots, n$ ,

\item[(e)] $\log^\delta(d)\; d^{\frac{\theta}{2}}\leq \lambda \leq
d^{\frac{\gamma}{2}}$, where $\delta\in(\frac{1}{2}, \infty)$ and
$\gamma\in(\theta, \alpha-\eta)$,
\end{enumerate}
then the thresholding estimator $\hat{u}^{\rm ST}_1$ is
consistent with $u_1$.\label{Th:02}
\end{theorem}

We offer a couple of remarks regarding the above theorem. First of
all, the theorem naturally reduces to Theorem \ref{Th:01} if we let
the sparsity index $\beta=0$. More importantly, this theorem, and
the following ones in Sections~\ref{sec:02} to~\ref{sec:04}, show
that the concepts depicted in Figure~\ref{fig:01} hold much more
generally than just under the conditions of
Example~\ref{example:01}. In particular, in the above
Theorem~\ref{Th:02}, setting $\theta=0$ and $\eta=\beta$ would give
the results plotted in Figure~\ref{fig:01}.

In addition, for different thresholding parameter $\lambda$, the ST
estimator $\hat{u}_1^{\rm ST}$ is consistent with $u_1$ with
different convergence rate. This result is stated in the following
theorem. The notation $\lambda= o(d^\rho)$ below means that $\lambda
d^{-\rho}\rightarrow 0$ as $d\rightarrow \infty$.
\begin{theorem}
For the thresholding parameter $\lambda= o(d^{\frac{\alpha-\eta-\varsigma}{2}})$,
where $\varsigma \in [0, \alpha-\eta-\theta)$, the corresponding thresholding estimator $\hat{u}^{\rm ST}_1$ is consistent
with $u_1$, with a convergence rate of $d^{\frac{\varsigma}{2}}$.
\label{corr:01}
\end{theorem}

\section{Asymptotic properties of RSPCA}\label{sec:03}

As noted in Section~\ref{sec:01}, several sparse PCA methods have
been proposed in the literature. Here we perform a detailed HDLSS
asymptotic analysis of the sparse PCA procedure developed by Shen
and Huang (2008)~\cite{shen2008sparse}. For simplicity, we refer to it as the
{\it regularized sparse PCA}, or RSPCA for short. All the detailed technical proofs are
again provided in Section~\ref{sec:07} and the supplement materials.

We start with briefly reviewing the methodological details of RSPCA.
(For more details, see ~\cite{shen2008sparse}.) Given
a $d$-by-$n$ data matrix $X_{(d)}$, consider the following penalized
sum-of-squares criterion:
\begin{eqnarray}
\| X_{(d)}-uv^T\|^2_F+P_{\lambda}(u),\quad {\rm subject\; to}\quad
\| v \|=1, \label{For:01}
\end{eqnarray}
where $u$ is a $d$-vector, $v$ is a unit $n$-vector, $\| \cdot \|_F$
denotes the Frobenius norm, and $P_{\lambda}(u)=\sum_{i=1}^d
p_{\lambda}\left(|{u}_{i,1}|\right)$ is a penalty function with
$\lambda\geq 0$ being the penalty parameter. The penalty function
can be any sparsity-inducing penalty. In
particular, Shen and Huang~\cite{shen2008sparse} considered the soft thresholding
(or $L_1$ or LASSO) penalty of Tibshirani (1996)~\cite{tibshirani1996regression}, the
hard thresholding penalty of Donoho and Johnstone (1994)~\cite{donoho1994ideal}, and the smoothly
clipped absolute deviation (SCAD) penalty of Fan and Li (2001)~\cite{fan2001variable}.

Without the penalty term or when $\lambda=0$, minimization
of~\eqref{For:01} can be obtained via singular value decomposition
(SVD)~\cite{eckart1936approximation}, which results in the best
rank-one approximation of $X_{(d)}$ as $\tilde{u}_1\tilde{v}^T_1$,
where $\tilde{u}_1$ and $\tilde{v}^T_1$ minimize the
criterion~\eqref{For:01}. The normalized $\tilde{u}_1$ turns out to
be the first empirical PC loading vector. With the penalty
term, Shen and Huang define the sparse PC loading vector as
$\hat{u}_1= {\tilde{u}_1}/{\| \tilde{u}_1 \|}$ where $\tilde{u}_1$
is now the minimizer of~\eqref{For:01} with the penalty term
included. The minimization now needs to be performed iteratively.
For a given $\tilde{v}_1$ in the criterion~(\ref{For:01}), we can
get the minimizing vector as
$\tilde{u}_1=h_{\lambda}\left(X_{(d)}\tilde{v}_1\right)$, where
$h_{\lambda}$ is a thresholding function that depends on the
particular penalty function used and the penalty (or thresholding)
parameter $\lambda$. See~\cite{shen2008sparse} for more details. The thresholding is
applied to the vector $X_{(d)}\tilde{v}_1$ componentwise.

Shen and Huang (2008)~\cite{shen2008sparse} proposed the following iterative procedure for minimizing the criterion~(\ref{For:01}):

\centerline{\fbox{\bf The RSPCA Algorithm}}\vskip -0.05in
\rule{5in}{.1mm}
\begin{enumerate}
\item Initialize:
    \begin{enumerate}
    \item Use SVD to obtain the best rank-one approximation
$\tilde{u}_1 \tilde{v}_1^T$ of the data matrix $X_{(d)}$,
 where $\tilde{v}_1$ is a unit vector.
 \item Set $\tilde{u}^{\rm old}_1=\tilde{u}_1$
 and $\tilde{v}^{\rm old}_1=\tilde{v}_1$.
    \end{enumerate}

\item Update:
    \begin{enumerate}
       \item $\tilde{u}^{\rm new}_1=h_{\lambda}\left(X_{(d)}\tilde{v}^{\rm old}_1\right)$.

          \item  $\tilde{v}^{\rm new}_1=\frac{X^T_{(d)}\tilde{u}^{\rm new}_1}{\| X^T_{(d)}\tilde{u}^{\rm new}_1\|}$.
      \end{enumerate}

   \item   Repeat Step 2 setting $\tilde{u}^{\rm old}_1=\tilde{u}^{\rm new}_1$ and $\tilde{v}^{\rm old}_1=\tilde{v}^{\rm new}_1$ until convergence.
    \item Normalize the final $\tilde{u}^{\rm new}_1$ to get $\hat{u}_1$, the desired sparse loading vector.
\end{enumerate}
\vspace*{-2mm} \rule{5in}{.1mm}


There exists a nice connection between the simple thresholding (ST)
method of Section~\ref{sec:02} and RSPCA. The ST estimator
$\hat{u}_1^{\rm ST}$ is exactly the sparse loading vector
$\hat{u}_1$ obtained from the first iteration of the RSPCA iterative
algorithm, when the hard thresholding penalty is used. In
particular, the first dual eigenvector $\tilde{v}_1$
in~(\ref{eq:dualeigevector}) is just the $\tilde{v}_1$ from the best
rank-one approximation $\tilde{u}_1 \tilde{v}_1^T$ of the data
matrix $X_{(d)}$. Then, the application of the simple thresholding
method to the vector $X_{(d)}\tilde{v}_1$ as in
(\ref{simplethreshold}) leads to the sparse ST estimator
$\hat{u}^{\rm ST}_1$. This is the same as applying the hard
thresholding penalty in~(\ref{For:01}) to generate the sparse
loading vector $\hat{u}_1$, for the given $\tilde{v}_1$. The
thresholding parameter $\lambda$ in (\ref{simplethreshold}) also
corresponds to the penalty parameter $\lambda$ in~(\ref{For:01}) in
the case of the hard thresholding penalty.

Below we develop conditions under which the sparse RSPCA loading vector $\hat{u}_1$ is consistent with
the population eigenvector $u_1$ when a proper thresholding parameter
$\lambda$ is used. All three of the soft thresholding, hard thresholding or SCAD penalties are considered. First, the following theorem states conditions when the first step sparse loading vector $\hat{u}_1$ is consistent with $u_1$ under the proper thresholding parameter $\lambda$.

\begin{theorem}  Under the assumptions and conditions  of Theorem~\ref{Th:02}, the first step
 sparse loading vector $\hat{u}_1$  is consistent with $u_1$.\label{Th:03}
\end{theorem}

Theorem~\ref{Th:03} explores conditions  when the first iteration of the iterative procedure of RSPCA
gives a consistent sparse
loading vector $\hat{u}_1$, with an appropriate thresholding parameter $\lambda$. Similar to the ST estimator, for
different parameters $\lambda$, $\hat{u}_1$ is consistent with $u_1$ with
different convergence rates. The result is given in the following Theorem~\ref{corr:02}.
\begin{theorem}
For the thresholding parameter $\lambda= o(d^{\frac{\alpha-\eta-\varsigma}{2}})$,
where $\varsigma \in [0, \alpha-\eta-\theta)$,
the sparse loading vector $\hat{u}_1$ in Theorem \ref{Th:03} is consistent
with $u_1$,  with a convergence rate of $d^{\frac{\varsigma}{2}}$.
\label{corr:02}
\end{theorem}


We then set $\hat{u}^{\rm old}_1$ to be the consistent sparse
loading vector obtained after the first iteration of the RSPCA
algorithm. We then obtain an updated estimate for $v_1$ as
$\tilde{v}^{\rm new}_1=X^T_{(d)} \hat{u}^{\rm old}_1/\| X^T_{(d)}
\hat{u}^{\rm old}_1\|$. The theorem below studies the asymptotic
properties of $\tilde{v}^{\rm new}_1$.
\begin{theorem}
Assume that $\hat{u}^{\rm old}_1$  is consistent with $u_1$ with the
convergence rate $d^{\frac{\varsigma}{2}}$, where
 $\varsigma \in [1-\alpha, \infty)$.
 If the  $\varepsilon$-condition is satisfied, then
 \begin{eqnarray*}
  \tilde{v}^{\rm new}_1 \xrightarrow{p} \frac{\tilde{W}_1}{\|\tilde{W}_1\|},
  \quad \mbox{as}\quad d\rightarrow \infty,
 \end{eqnarray*}
 where $\tilde{W}_1=(w_{1,1},\cdot\cdot\cdot,w_{1,n})$ follows a standard $n$-dimensional
 normal distribution $N(0, I_n)$ and the $w_{i,j}$ are defined in (\ref{eq:w}).\label{Th:pc}
 \end{theorem}

Since Theorem~\ref{Th:pc} establishes the asymptotic properties of
$\tilde{v}^{\rm new}_1$, we can now study the asymptotic properties
of the updated sparse loading vector
\begin{eqnarray}
\hat{u}^{\rm new}_1=\frac{\tilde{u}^{\rm new}_1}{\| \tilde{u}^{\rm
new}_1\|},\quad {\rm with}\quad \tilde{u}^{\rm
new}_1=h_{\lambda}(X_{(d)}\tilde{v}_1^{\rm new}), \label{u_rSVD}
\end{eqnarray}
as defined in the iterative procedure of RSPCA. The following
Theorem~\ref{Th:04} shows that with a proper choice of the thresholding parameter $\lambda$, the updated sparse loading vector $\hat{u}^{\rm
new}_1$ remains to be consistent with the population eigenvector
$u_1$.
\begin{theorem} Under the assumptions and conditions of Theorems~{\ref{Th:02}} and~\ref{Th:pc},
the updated sparse loading vector $\hat{u}^{\rm new}_1$ in
(\ref{u_rSVD}) is consistent with $u_1$.\label{Th:04}
\end{theorem}

For different threshold parameters $\lambda$, $\hat{u}^{\rm new}_1$
is again consistent with $u_1$ with different convergence rates, as
seen in the following theorem.
\begin{theorem}
For the thresholding parameter $\lambda=
o(d^{\frac{\alpha-\eta-\varsigma}{2}})$, where $\varsigma \in [0,
\alpha-\eta-\theta)$, the updated sparse loading vector
$\hat{u}^{\rm new}_1$ in Theorem \ref{Th:04} is consistent with
$u_1$,  with convergence rate $d^{\frac{\varsigma}{2}}$
.\label{corr:03}
\end{theorem}

According to Theorems~\ref{corr:02} and~\ref{corr:03}, if
$\alpha-\eta-\theta> 1-\alpha$, then we can choose the thresholding parameter $\lambda=o(d^{\frac{\alpha-\eta-\varsigma}{2}})$ and make
the updated sparse loading vector $\hat{u}_1^{\mbox{\rm new}}$ in
(\ref{u_rSVD}) to be consistent with $u_1$ at every updating step.

\section{Strong Inconsistency}\label{sec:04}

We have shown that we can attain consistency using sparse PCA, when
the spike index $\alpha$ is greater than the sparsity index $\beta$.
This motivates the question of consistency using sparse PCA when the
spike index $\alpha$ is smaller than the sparsity index $\beta$. To
answer this question, we consider an oracle estimator which uses
the exact positions of zero entries of the population eigenvector
$u_1$. We will show that even this oracle estimator is strongly
inconsistent with the population eigenvector $u_1$ when the spike
index $\alpha$ is smaller than the sparsity index $\beta$. Compared
with this oracle sparse PCA, threshold methods can perform no better
because they also need to estimate location of the zero entries;
hence threshold methods will also be strongly inconsistent.

For Example~\ref{example:01}, the first $\lfloor d^\beta \rfloor$
entries of the population eigenvector $u_1$ are known to be the
non-zero entries. So we could first find a $\lfloor d^\beta
\rfloor$-dimensional estimator $\hat{u}^*_1$ through subspace PCA
for the $\lfloor d^\beta \rfloor$-dimensional subspace eigenvector
$u^*_1$ which is proportional to the following $\lfloor d^\beta
\rfloor$-dimensional vector,
\begin{eqnarray*}
\ddot{u}_1=( \overbrace{1,\ldots,1}^{\lfloor d^{\beta}\rfloor})^T.
\end{eqnarray*}
Then we get the oracle (OR) estimator for $u_1$ as,
\begin{eqnarray*}
\hat{u}^{\rm OR}_1=((\hat{u}^*_1)^T,
\overbrace{0,\ldots,0}^{d-\lfloor d^{\beta}\rfloor})^T.
\end{eqnarray*}
The oracle estimator $\hat{u}^{\rm OR}_1$ has the same sparsity as
the population eigenvector $u_1$. Furthermore, it is strongly
inconsistent with $u_1$ when $\alpha<\beta$.

To make this precise, we study the procedure  to generate the oracle
estimator for general single component models. Assume that the first
$\lfloor d^\beta \rfloor$ entries of the population eigenvector
$u_1$ are non-zero and the rest are all zero:
$u_1=(u_{1,1},\ldots,u_{d,1})^T$,where $u_{i,1}\neq 0,
i=1,\ldots, \lfloor d^\beta \rfloor$, otherwise $u_{i,1}=0$. Let
$X^*_j=(x_{1,j},\ldots,x_{\lfloor d^\beta \rfloor,j})^T\sim
N(0,\Sigma^*_{\lfloor d^\beta\rfloor})$, where $\Sigma^*_{\lfloor
d^\beta\rfloor}$ is the covariance matrix of $X^*_j$,
$j=1,\ldots,n$. Then, the eigen-decomposition of $\Sigma^*_{\lfloor
d^\beta\rfloor}$ is
\begin{eqnarray*}
\Sigma^*_{\lfloor d^\beta\rfloor}=U^*_{\lfloor d^\beta\rfloor}\Lambda^*_{\lfloor d^\beta\rfloor}(U^*_{\lfloor d^\beta\rfloor})^T,
\end{eqnarray*}
where $\Lambda^*_d$ is the diagonal matrix of eigenvalues
$\lambda^*_1\geq\lambda^*_2\geq\ldots\geq\lambda^*_{\lfloor
d^\beta\rfloor}$ and $U^*_{\lfloor d^\beta\rfloor}$ is the matrix of
the corresponding eigenvectors so that $U^*_{\lfloor
d^\beta\rfloor}=[u^*_1,\ldots,u^*_{\lfloor d^\beta\rfloor}]$. Since
the last $d-\lfloor d^\beta\rfloor$ entries of the first population
eigenvector $u_1$ equal  zero, it follows that the first eigenvector
$u^*_1$ of $\Sigma^*_{\lfloor d^\beta \rfloor}$  is formed by the
non-zero entries of the population eigenvector $u_1$, i.e.
$u^*_1=(u_{1,1},\ldots,u_{\lfloor d^\beta \rfloor,1})^T$. So we
have
\begin{eqnarray*}
u_1=((u^*_1)^T, \overbrace{0,\ldots,0}^{d-\lfloor d^{\beta}\rfloor})^T.
\label{subspace:population}
\end{eqnarray*}

Consider the following data matrix $X^*_{\lfloor d^\beta
\rfloor}=[X^*_1,\ldots,X^*_n]$, and denote the sample covariance
matrix by $\hat{\Sigma}^*_{\lfloor
d^\beta\rfloor}=n^{-1}X^*_{\lfloor d^\beta\rfloor }X_{\lfloor
d^\beta \rfloor}^{T}$. Then, the sample covariance matrix
$\hat{\Sigma}^*_{\lfloor d^\beta\rfloor}$ can be similarly
decomposed as
\begin{eqnarray*}
\hat{\Sigma}^*_{\lfloor d^\beta \rfloor}=\hat{U}^*_{\lfloor d^\beta \rfloor}
\hat{\Lambda}^*_{\lfloor d^\beta \rfloor}(\hat{U}^*_{\lfloor d^\beta \rfloor})^T,
\end{eqnarray*}
where $\hat{\Lambda}^*_{\lfloor d^\beta \rfloor}$ is the diagonal
matrix of the sample eigenvalues
$\hat{\lambda}^*_1\geq\hat{\lambda}^*_2\geq\ldots\geq\hat{\lambda}^*_{\lfloor
d^\beta \rfloor}$ and $\hat{U}^*_{\lfloor d^\beta \rfloor}$ is the
matrix of the corresponding sample eigenvectors so that
$\hat{U}^*_{\lfloor d^\beta
\rfloor}=[\hat{u}^*_1,\ldots,\hat{u}^*_d]$. Then, we define the
oracle (OR) estimator as
\begin{eqnarray}
\hat{u}^{\rm OR}_1=((\hat{u}^*_1)^T,
\overbrace{0,\ldots,0}^{d-\lfloor d^{\beta}\rfloor})^T.
\label{subspace}
\end{eqnarray}

The following theorem states the main result regarding strong
inconsistency.
\begin{theorem}  Assume that
$X_1,\ldots,X_n$ are random samples from a
$d$-dimensional normal distribution $ N(0, \Sigma_d) $. The first
population eigenvector is $u_1=(u_{1,1},\ldots,u_{d,1})^T$,where
$u_{i,1}\neq 0, i=1,\ldots, \lfloor d^\beta \rfloor$, otherwise
$u_{i,1}=0$. If the following
conditions are satisfied:

(a) $\lambda_1\sim d^\alpha $, $\lambda_2 \sim d^\theta $, $\lambda_d\sim 1 $ and $\sum_{i=2}^d\lambda_i\sim d $,
where $\theta \in [0, \frac{\beta}{2})$,

(b) $\alpha<\beta$,

then the oracle estimator $\hat{u}^{\rm OR}_1$ in (\ref{subspace})
is strongly inconsistent with $u_1$.\label{inconsistent}
\end{theorem}

\section{Simulations for sparse PCA}\label{sec:05}
Here, we will perform simulation studies to illustrate the performance of the ST method and the RSPCA with the hard thresholding penalty.
Let the sample size $n=25$ and the dimension $d=10, 000$.  To generate the data matrix $X_{(d)}$, we first need to construct the population covariance matrix for $X_{(d)}$ that approximates the conditions of Theorems~\ref{Th:02} and~\ref{Th:03} when the spike index $\alpha$ is greater
than the sparsity index $\beta$.

For the population covariance matrix, we consider the motivating model in Example~\ref{example:01} for the first population eigenvector and the eigenvalues, where the first eigenvalue $\lambda_1=d^\alpha$ and the rest equal one, i.e.
$\lambda_i=1$, $i\geq2$. For the additional population
 eigenvectors $u_i$, $2\leq i\leq \lfloor d^{\beta}\rfloor$, let the last $d-\lfloor d^{\beta}\rfloor$
 entries of these eigenvectors be zero. In particular, let the
 eigenvectors $u_i$, $2\leq i\leq \lfloor d^{\beta}\rfloor$, be proportional to
 \begin{equation*}
\dot{u}_i=( \overbrace{1,\ldots,1}^{i-1}, -i+1, 0,\ldots,0)^T.
\end{equation*}
After normalizing $\dot{u}_i$, we get the $i$-th eigenvector
$u_i={\dot{u}_i}/{\|\dot{u}_i\|}$. For $i > \lfloor
d^{\beta}\rfloor$, let the $i$-th eigenvector have just one non-zero
entry in the $i$-th position such that $u_i=(
\overbrace{0,\ldots,0}^{i-1}, 1, 0,\ldots,0)^T$.

Then the data matrix is generated as
\begin{equation*}
X_{(d)}= d^{\frac{\alpha}{2}} u_1 z^T_1+
\sum_{i=2}^d u_i z^T_i,
\end{equation*}
where the $z_i$ are generated from the $n$-dimensional standard normal distribution $N(0, I_n)$.

We select twenty spike and sparsity pairs $(\alpha, \beta)$ that have spike index $\alpha= \{0.2, 0.4, 0.6, 0.8\}$
and sparsity index $\beta=\{0, 0.1, 0.3, 0.5, 0.7\}$, which are shown in Figure~\ref{fig:01}. We perform the simulation
for all twenty spike and sparsity pairs. For each spike and sparsity pair $(\alpha, \beta)$, we generate 100 realizations of the data matrix $X_{(d)}$. Results for three representative pairs are reported below, and interesting observations are discussed. Additional simulation results can be found at~\cite{shen2011online}.\\

First of all, the plots in Figure~\ref{fig2} summarize the results for the
spike and sparsity pair $(\alpha, \beta)=(0.6, 0.1)$,
corresponding to one of the square dots in the white (consistent) triangular area of Figure~\ref{fig:01}.
For each replication of the data matrix $X_{(d)}$ and a range of the thresholding parameter $\lambda$,
we obtained the ST estimator $\hat{u}^{\rm ST}_1$ (Section~\ref{sec:02})
and the RSPCA estimator $\hat{u}_1$ (Section~\ref{sec:04}). Then we calculate the angle between the estimates $\hat{u}^{\rm ST}_1$ (or $\hat{u}_1$)
and the first population eigenvector $u_1$ through (\ref{Consistency}).
Plotting this angle as a function of the thresholding parameter $\lambda$ gives the curve in Panel (A) of Figure~\ref{fig2}.
Since ST and RSPCA have very similar performance in this case, we just show the RSPCA plots in Figure~\ref{fig2}. The 100 simulation realizations of the data matrices $X_{(d)}$ generate the one-hundred curves in the panel. We rescale the thresholding parameter $\lambda$ as  $\mbox{log}_{10}(\lambda+10^{-5})$,
to help reveal clearly the tendency of the angle curves as the thresholding parameter increases.

\begin{figure}[ht!H]
\vspace{-.5cm}
     \begin{center}
     \advance\leftskip-1.75cm
     \includegraphics[width=16.05cm]{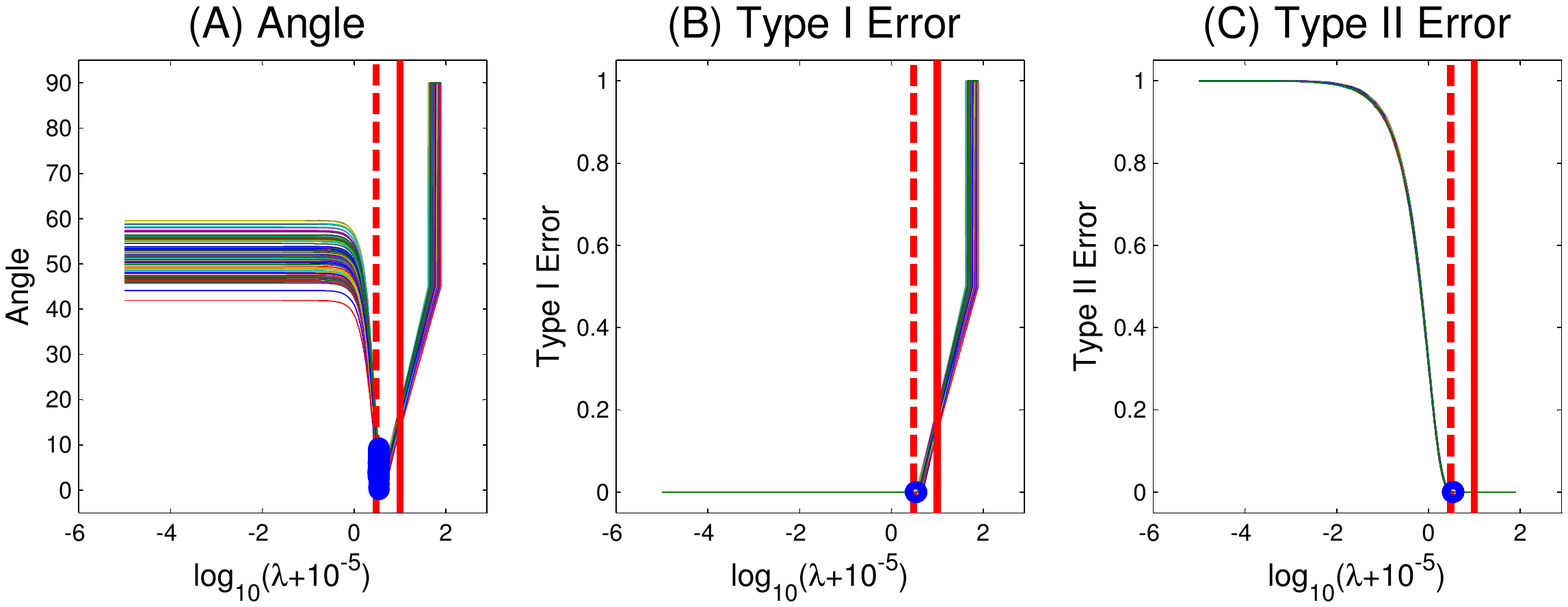} \vspace{-6.5cm}
     \end{center}
      \caption{Performance summary of RSPCA for  spike index $\alpha= 0.6 $ and
      sparsity index $\beta=0.1$ where consistency is expected. Panel (A) shows angle
      to the first population eigenvector as a function of thresholding parameter $\lambda$. Panel (B)
      and (C) are Type I Error  and Type II Error as a function of $\lambda$. The vertical dashed and solid lines are the left  and right
      bounds of the range of the thresholding  parameter, which leads to the consistency of RSPCA. These show very good performance of
      RSPCA within the indicated range, which empirically confirms our asymptotic calculation. The circles indicate values at the BIC choice of $\lambda$.}
       \label{fig2}
\vspace{-.25cm}
     \end{figure}

In these angle plots,
the angles with $\lambda=0$ (essentially the left edge of each plot) correspond to the ones obtained by the conventional PCA. Note that these angles are all over 40 degrees which confirms the results of Jung and Marron (2009)~\cite{jung2009pca} that
  when the spike index $\alpha<1$, the conventional PCA can not generate a consistent estimator for the population eigenvector $u_1$. As $\lambda$ increases, the angle remains stable for a while, then decreases to almost 0 degree, before eventually starting to increase to 90 degrees. The dashed and solid vertical lines  in the angle plots indicate the range of the thresholding parameter that gives a consistent estimator for $u_1$, as stated in Theorems~\ref{Th:02} and~\ref{Th:04}. These plots suggest that RSPCA does improve over PCA and the indicated thresholding range is very reasonable in this case, which in turn empirically validates the
  asymptotic results of the theorems. For each realization of the data, as
   the thresholding parameter increases, all entries will be thresholded out, i.e. become zero, so
    the sparse PCA estimator eventually becomes a  $d$-dimensional zero vector. Hence the angles go to 90 degrees when the thresholding parameter is large enough.

Zou, Hastie and Tibshirani (2007)~\cite{zou2007degrees} suggest the use of the Bayesian Information Criterion (BIC)~\cite{schwarz1978estimating} to select the number of the non-zero coefficients for a lasso regression. Lee et al. (2010)~\cite{lee-biclustering} apply this idea to the sparse PCA context. According to~\cite{lee-biclustering}, for a fixed $\tilde{v}_1$, minimization of (\ref{For:01}) with respect to $\tilde{u}_1$ is equivalent to minimization of the following penalized regression criterion with respect to
$\tilde{u}_1$:
\begin{eqnarray}
\|X_{(d)}-\tilde{u}_1\tilde{v}_1^T\|^2_F+P_{\lambda}(\tilde{u}_1)=
\|Y-(I_d\bigotimes
\tilde{v}_1)\tilde{u}_1\|^2+P_{\lambda}(\tilde{u}_1),\label{penalized
regression}
\end{eqnarray}
where $Y=(X^1, \ldots, X^d)^T$, with $X^i$ being the $i$-th row of $X_{(d)}$, and $\bigotimes$ is the Kronecker product.
Following their suggestion, for
the above penalized regression (\ref{penalized regression}) with a fixed $\tilde{v}_1$, we define
\begin{eqnarray}
\mbox{BIC}(\lambda)=\frac{\|Y-\hat{Y}\|^2}{nd
\hat{\sigma}^2}+\frac{\log (nd)}{nd} \hat{df}(\lambda), \label{BIC}
\end{eqnarray}
where $\hat{\sigma}^2$ is the ordinary-least squares estimate of the error variance, and $\hat{df}(\lambda)$ is the degree of sparsity for the thresholding parameter $\lambda$, i.e. the number of non-zero entries in $\tilde{u}_1$. For every step of the iterative procedure of RSPCA,
we can use BIC (\ref{BIC}) to select the thresholding parameter and then obtain the corresponding sparse PC direction, until the algorithm converges.



For every  angle curve in the angle plots of Figure~\ref{fig2}, we use a blue circle to indicate the thresholding parameter $\lambda$ that is selected by BIC during the last iterative step of RSPCA, and the corresponding angle. In the current $\alpha=0.6$, $\beta=0.1$ context, BIC works well, and all the BIC-selected $\lambda$ values are very close, so the 100 circles are essentially over plotted on each other. BIC also
  works well for the other spike and sparsity pairs $(\alpha, \beta)$ we considered where $\alpha>\beta$, which are shown in~\cite{shen2011online}.

Another measure of the success of a sparse estimator is in terms of
 which entries are zeroed. Type I Error is  the proportion of non-zero entries in $u_1$ that are mistakenly estimated as zero. Type II Error
  is the proportion of zero entries in $u_1$ that are mistakenly estimated as non-zero.
Similar to the angle, Type I Error  (Type II Error) is also a function
of the thresholding parameter. For each replication of the data matrix $X_{(d)}$, we calculate a Type I Error (Type II Error) curve.
Thus, there are one hundred such curves in Panels (B) and (C) of Figure~\ref{fig2}, respectively. The dashed and solid lines in these two panels are the same as those in Panel (A). Note that for all the thresholding parameters in the range indicated by the lines, the errors are
very small, which is again consistent with the asymptotic results of Theorems~\ref{Th:02} and~\ref{Th:04}. Again, the circles in these plots are selected by BIC  and they have the same horizonal thresholding parameter, as in the angle plots.
Thus, BIC works well here. BIC also generates similarly very small errors for the other spike and sparsity pairs
$(\alpha, \beta)$ in Figure~\ref{fig:01} that satisfy $\alpha>\beta$.\\

Next we will compare the relative performance among PCA, ST and  RSPCA.
In almost all cases, ST and RSPCA give better results than PCA
and in some extreme cases, the three methods have similar poor performance.
Although in most cases both ST and RSPCA have similar performance,
however, there are some cases (for example when $\alpha =0.4 $ and $\beta=0.3$),
where RSPCA performs better than ST. For every replication of the data matrix $X_{(d)}$, we use BIC to
select the thresholding parameter, and then calculate the ST estimator $\hat{u}^{\rm ST}_1$
and the RSPCA estimator $\hat{u}_1$. After that, we calculate the angle, Type I Error and Type II Error for the three estimators, as well as the difference
between ST and RSPCA (ST minus RSPCA). For each measure, the 100 values are summarized using box plots
in Figure~\ref{fig:03}.

 \begin{figure}[ht!H]
 \vspace{-.5cm}
     \begin{center}
     \advance\leftskip-1.95cm
     \includegraphics[width=16.35cm]{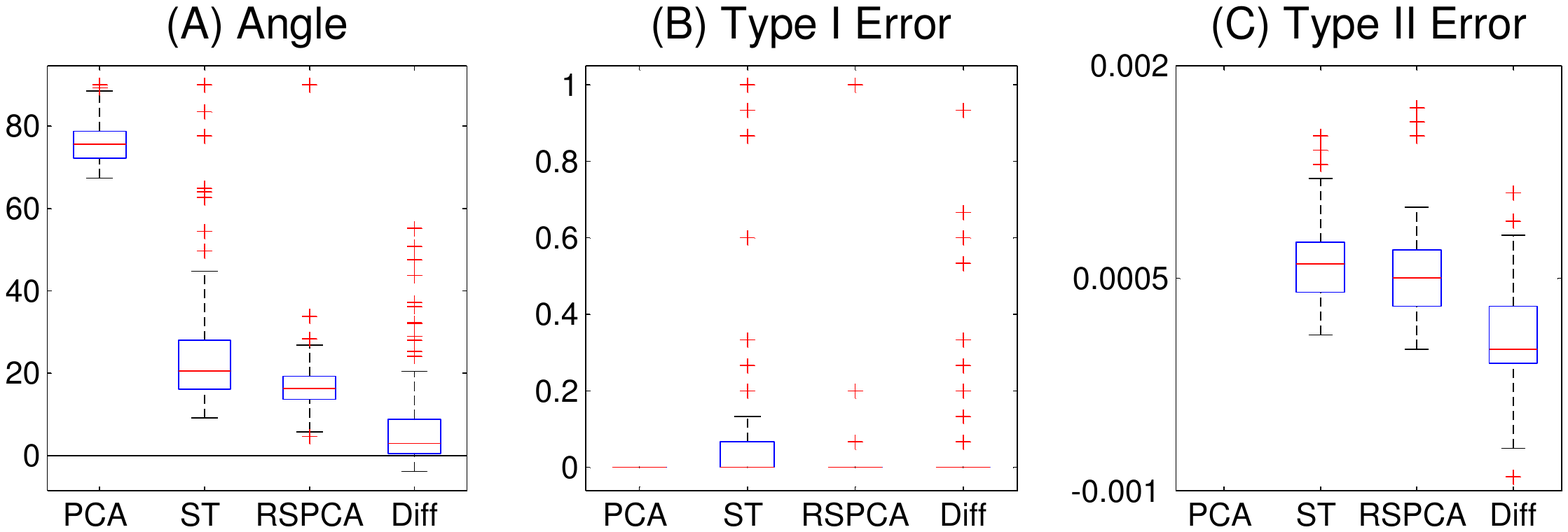} \vspace{-7cm}
     \end{center}
     \caption{Comparison of PCA, ST and RSPCA for spike index $\alpha= 0.4 $ and
     sparsity index $\beta=0.3$. Panels (A), (B) and (C) respectively contain four angle, Type I Error and Type II Error
      box plots: (i) conventional PCA; (ii) and (iii) ST and RSPCA with BIC; (iv) the difference between ST and RSPCA. In Panel (A), angles for conventional PCA are generally
      larger than ST and RSPCA which
       indicates the worse performance of contentional PCA. In addition, the angles and Type I Errors for ST are larger than RSPCA and their difference box plots furthermore confirm this point, which indicates the better performance of RSPCA in this case. Type II Errors for ST and RSPCA are almost the same.}
     \label{fig:03}
     \vspace{-.25cm}
\end{figure}

Panel (A) of Figure~\ref{fig:03} shows the box plots of the angles between the first population eigenvector $u_1$ and the estimates
obtained by PCA, ST and RSPCA, as well as the differences between ST and RSPCA. Note that the PCA angles are large, compared with
ST and RSPCA, indicating the worse performance of PCA.  The angle of ST seems larger than RSPCA. For a deeper view
of this comparison, the pairwise differences are studied in the fourth box plot of the panel. The angle differences are almost
always positive, with some differences bigger than 50 degrees, which suggests that RSPCA has a better performance than ST.
Similar conclusions can be made from the box plots of the errors, in Panels (B) and (C) of Figure~\ref{fig:03}.
The box plot for PCA is not shown in Panel (C) because the corresponding Type II Error almost always equals one, which
is far outside the shown range of interest.\\

\begin{figure}[ht!H]
\vspace{-.5cm}
     \begin{center}
     \advance\leftskip-1.75cm
     \includegraphics[width=16.05cm]{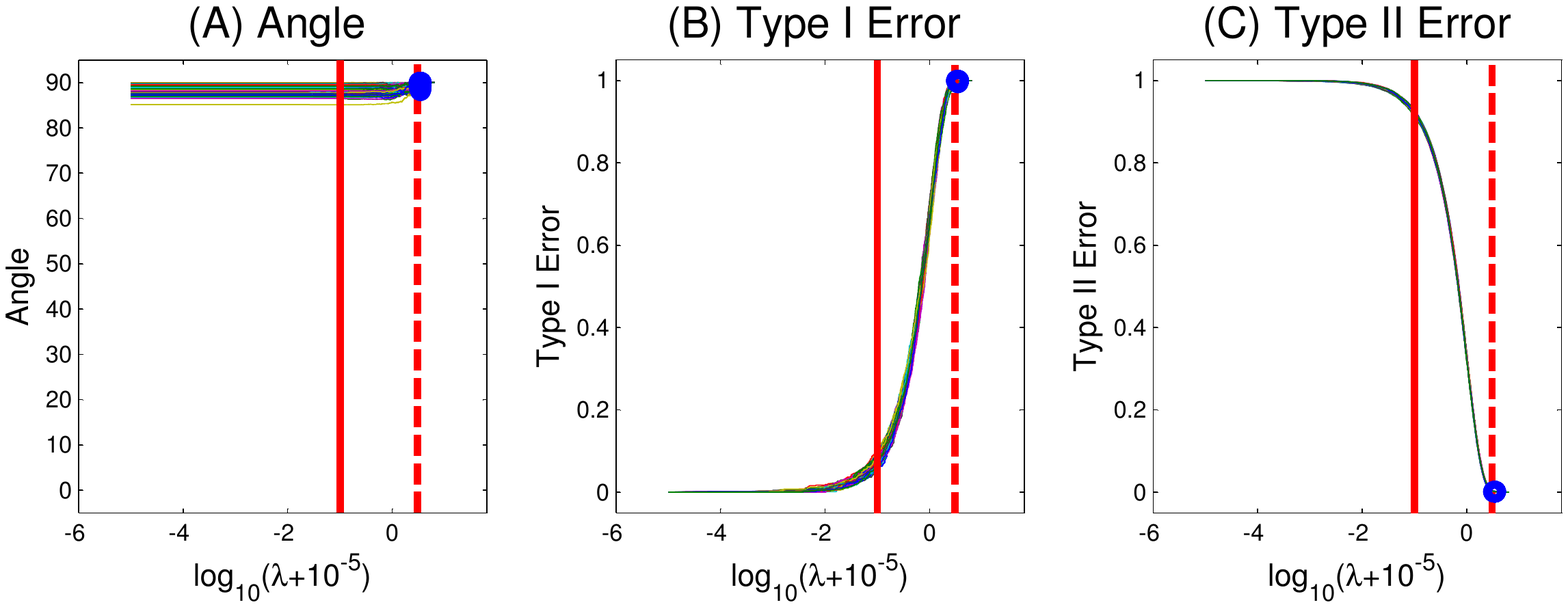} \vspace{-6.5cm}
     \end{center}
      \caption{Performance summary of RSPCA for spike index $\alpha= 0.2 $ and
      sparsity index $\beta=0.7$ where strong inconsistency is expected. Panel (A) shows the angle curves, between the RSPCA estimator and
      the first population eigenvector. Same format as Figure~\ref{fig2}. Since the spike index $\alpha=0.2$ is smaller
       than the sparsity index $\beta=0.7$, it follows that the right bound (solid line) is smaller than the left bound (dashed line).
       Thus the theorems do not give a meaningful range of the thresholding parameter.
        As expected, performance is very poor for any thresholding parameter $\lambda$.}
\vspace{-.25cm}
\label{fig:04}
     \end{figure}

Finally, Theorems~\ref{Th:02} and~\ref{Th:04} consider the condition that the spike index $\alpha$ is greater
than the sparsity index $\beta$. When $\alpha$ is smaller than $\beta$,
neither ST nor RSPCA is expected to give consistent estimation for the first population eigenvector $u_1$,
as discussed in Section~\ref{sec:04}. For the spike and sparsity pairs $(\alpha, \beta)$ such that $\alpha < \beta$, the simulation results also confirm this point. Here, we display the simulation plots for the spike and sparsity pair $(\alpha, \beta)=(0.2, 0.7)$  in Figure~\ref{fig:04}
as a representative of such simulations. Since ST and RSPCA have very similar performance here, we just
show the simulation results for RSPCA.
Similar to Figure~\ref{fig2}, the circles in Figure~\ref{fig:04} correspond to the thresholding parameter selected by BIC.
From the angle plots, we can see that the angles, selected by BIC, are close to 90 degrees, which suggests the failure of BIC in this case. In fact, all the angle curves are above 80 degrees. Thus, neither
ST nor RSPCA generates a reasonable sparse estimator. This is a common phenomenon when the spike index $\alpha$ is smaller than the sparsity index $\beta$.
It is consistent with the theoretical investigation in Section~\ref{sec:04}.

Furthermore, the corresponding Type I Error, generated by ST or RSPCA with BIC, is close to one. This further confirms that
BIC doesn't work when the spike index $\alpha$ is smaller than the sparsity index $\beta$. ST and RSPCA
 with $\lambda=0$ is just the conventional PCA, and typically will not generate a sparse estimator.
This entails that the Type I Error and Type II Error, corresponding to $\lambda=0$, respectively equals zero and
one. As the thresholding parameter increases, more and more
entries are thresholded out; hence Type I Error  increases to one and Type II Error  decreases to zero.

\section{Non-Gaussian Variations}\label{sec:06}\vspace{-.5cm}
In this paper, we consider HDLSS data analysis contexts, using the high dimensional normal
distribution. In the future, we hope to extend our theorems to more general distributions. However,
this will be challenging because  sparse PCA
methods may not work in some extreme cases. This point is illustrated by the following interesting example.\\

\begin{example} Let $\alpha\in(0,1)$ and  $X=(x_1,\ldots, x_d)^T$, where $\{x_i, i=1, \ldots, d\}$ are
independent discrete random variables with the following discrete probability distributions:
$$x_1=\begin{cases}
d^{\frac{\alpha}{2}}, \quad & {\rm with\; probability}\;\;\frac{1}{2},\\
-d^{\frac{\alpha}{2}}, \quad & {\rm with\; probability}\;\;\frac{1}{2};
\end{cases}
$$
and for $i=2
\ldots,d$
$$x_i=\begin{cases}
d^{\frac{\alpha+1}{4}}, \quad & {\rm with\; probability}\;\;d^{-\frac{\alpha+1}{2}},\\
-d^{\frac{\alpha+1}{4}}, \quad & {\rm with\; probability}\;\;d^{-\frac{\alpha+1}{2}},\\
0, \quad & {\rm with\; probability}\;\;1-2d^{-\frac{\alpha+1}{2}}.
\end{cases}
$$

Then $X$ has mean $0$ and variance-covariance $\Sigma_d$ with
\begin{eqnarray*}
\Sigma_d=d^\alpha u_1u_1^T+ \sum_{k=2}^d
u_ku_k^T,
\end{eqnarray*}
where $u_1=(1,0,\ldots,0)^T$.
\end{example}

Suppose that we only have sample  size $n=1$, i.e. $X_1=(x_{i,1},
\ldots, x_{d,1})^T$, then the first empirical eigenvector
\begin{eqnarray*}
\hat{u}_1=\left(\hat{u}_{1,1},\ldots,\hat{u}_{d,1}\right)^T=\frac{1}{\sqrt{\sum_{i=1}^d
x^2_{i,1} }} ( x_{i,1}, \ldots,x_{i,d})^T.
\end{eqnarray*}
Under this condition, we have
\begin{eqnarray*}
\nonumber P\left({\mbox{argmax}}_i |\hat{u}_{i,d}|=1\right)&=&P\left(|x_{1,1}|>
\mbox{max}\left\{|x_{2,1}|,\ldots,|x_{d,1}|\right\}\right)\\ \nonumber
&=&P\left(x_{2,1}=0, \ldots, x_{d,1}=0\right)\\ \nonumber
&=&\left(1-2d^{-\frac{\alpha+1}{2}}\right)^{d-1}\\
&\longrightarrow& 0 \; as \; d\rightarrow\infty.
\end{eqnarray*}
In particular, the absolute value of the first entry of the empirical eigenvector can not be greater than the others
with probability 1, so we can not always
threshold out the right entries which results in the failure of the simple thresholding method. Similar considerations
apply to other sparse  PCA methods.

 \section{Proofs}\label{sec:07}

 \subsection{Proofs of Theorem~\ref{Th:02} and Theorem~\ref{corr:01}}\label{subsec:71}

In order to prove Theorem~\ref{Th:02} and Theorem~\ref{corr:01}, we need
the dependent extreme value results from Leadbetter, Lindgren and Rootzen (1983)~\cite{leadbetter1983extremes}, in particular their Lemma 6.1.1 and Theorem 6.1.3.

An immediate consequence of those results is the following proposition.
\begin{proposition} Suppose that the
standard normal sequence $\{\xi_i, i=1,\dots, \lfloor d^\beta \rfloor\}$ satisfies the mixing condition (\ref{mixing}).
Let the positive constants $\{c_i\}$ be such that $\sum_{i=1}^{\lfloor d^\beta \rfloor}
(1-\Phi(c_i))$ is bounded and such that $C_{\lfloor d^\beta \rfloor}=\mbox{min}_{1\leq i\leq {\lfloor d^\beta \rfloor}}$
$c_i\geq c(log  (\lfloor d^\beta \rfloor))^{\frac{1}{2}}$ for  some $c>0$.

Then the following holds:
\begin{equation*}
P\left[\bigcap_{i=1}^{\lfloor d^\beta \rfloor}\left\{\xi_i\leq c_i \right\}\right]-\prod_{i=1}^{\lfloor d^\beta \rfloor}
\Phi(c_i)\longrightarrow 0, as\; d\rightarrow\infty,
\end{equation*}
where $\Phi$ is the standard normal distribution function. Furthermore, if for some $\jmath\geq0$, we have
\begin{equation*}
\sum_{i=1}^{\lfloor d^\beta \rfloor} (1-\Phi(c_i))\longrightarrow \jmath, as\;
d\rightarrow\infty,
\end{equation*}
then
\begin{equation*}
P\left[\bigcap_{i=1}^{\lfloor d^\beta \rfloor} \left\{\xi_i\leq c_i\right\}\right]\longrightarrow
e^{-\jmath}, as\; d\rightarrow\infty.
\end{equation*}\label{pro:01}
\end{proposition}
Proposition~\ref{pro:01} is used to control the right side of (\ref{eq:h}) through the following lemma.
\begin{lemma} Suppose that $\xi_i\sim N(0, \delta_{i,i})$ satisfies the mixing condition (\ref{mixing}), where
$\delta_{ij}$ is the covariance of the normal sequence $\{\xi_i\}$, $i, j=1,\ldots,\lfloor d^\beta \rfloor$.
 If
$C_{\lfloor d^\beta \rfloor}\geq (\log (\lfloor d^\beta \rfloor))^{\delta}\mbox{max}_{1\leq i\leq
\lfloor d^\beta \rfloor}\delta^{\frac{1}{2}}_{ii}$, where $\delta \in(\frac{1}{2},\infty)$,
then
\begin{eqnarray*}
C_{\lfloor d^\beta \rfloor}^{-1} \mbox{max}_{1\leq i\leq \lfloor d^\beta \rfloor} |\xi_i|\xrightarrow{p} 0, as
\;d\rightarrow \infty .
\end{eqnarray*}\label{lem:01}
\end{lemma}

\begin{proof} Note that for every $\tau>0$
\begin{eqnarray}\label{eq:02}
 && P\left[ C_{\lfloor d^\beta \rfloor}^{-1} \mbox{max}_{1\leq i\leq \lfloor d^\beta \rfloor} |\xi_i|
>\tau \right]= P\left[ \mbox{max}_{1\leq i\leq \lfloor d^\beta \rfloor} |\xi_i|
> C_{\lfloor d^\beta \rfloor}\tau \right] \\
 \nonumber&& \leq P\left[ \left\{\mbox{max}_{1\leq i\leq \lfloor d^\beta \rfloor}  \xi_i
> C_{\lfloor d^\beta \rfloor}\tau \right\}\bigcup\left\{\mbox{max}_{1\leq i\leq \lfloor d^\beta \rfloor} (-\xi_i)
> C_{\lfloor d^\beta \rfloor}\tau \right\}\right] \\ \nonumber
&&\leq P\left[ \mbox{max}_{1\leq i\leq \lfloor d^\beta \rfloor} \xi_i
> C_{\lfloor d^\beta \rfloor}\tau \right] +P\left[ \mbox{max}_{1\leq i\leq {\lfloor d^\beta \rfloor}} (-\xi_i)
> C_{\lfloor d^\beta \rfloor}\tau \right]\\ \nonumber
&&= 2P\left[ \mbox{max}_{1\leq i\leq \lfloor d^\beta \rfloor} \xi_i
> C_{\lfloor d^\beta \rfloor}\tau \right]    \\ \nonumber
&&\leq 2\left(1-P\left[ \bigcap_{i=1}^{\lfloor d^\beta \rfloor}
\left\{\xi_i \delta_{ii}^{-\frac{1}{2}} \leq c (log (\lfloor d^\beta \rfloor))^{\delta}\right\}\right]\right),
\end{eqnarray}
where $c$ is a positive constant. Since
\begin{equation*}
\sum_{i=1}^{\lfloor d^\beta \rfloor} \left(1-\Phi\left(c
(\log (\lfloor d^\beta \rfloor))^{\delta}\right)\right)\longrightarrow 0, as\;
d\rightarrow\infty,
\end{equation*}
it then follows from Proposition~\ref{pro:01} that
\begin{eqnarray}
P\left[ \bigcap_{i=1}^{\lfloor d^\beta \rfloor} \left\{\xi_i \delta_{ii}^{-\frac{1}{2}} \leq c
(\log (\lfloor d^\beta \rfloor))^{\delta}\right\}\right]\longrightarrow 1, as\;
d\rightarrow\infty.\label{eq:03}
\end{eqnarray}
From (\ref{eq:02}) and (\ref{eq:03}), we can get
\begin{eqnarray*}
C_{\lfloor d^\beta \rfloor}^{-1} \mbox{max}_{1\leq i\leq \lfloor d^\beta \rfloor} |\xi_i|\xrightarrow{p} 0, as
\;d\rightarrow \infty.
\end{eqnarray*}
\end{proof}

Now we will begin the proof of Theorem~\ref{Th:02} and Theorem~\ref{corr:01}.  Denote
$\tilde{X}_i=(x_{i,1},\cdot\cdot\cdot,x_{i,n})^T$,
$\tilde{Z}_i=(z_{i,1},\cdot\cdot\cdot,z_{i,n})^T$,
$\tilde{W}_i=(w_{i,1},\cdot\cdot\cdot,w_{i,n})^T$and
$\tilde{H}_i=(h_{i,1},\cdot\cdot\cdot,h_{i,n})$, $i=1,\cdot\cdot\cdot,d$.

Note that
\begin{eqnarray}
|<\hat{u}^{\rm ST}_1, u_1>|=\frac{|\sum_{i=1}^{[d^\beta]}
\breve{u}_{i,1}u_{i,1}
|}{\sqrt{\sum_{i=1}^{d}(\breve{u}_{i,1})^2 }}=\frac{\lambda_1^{-\frac{1}{2}}|
\sum_{i=1}^{[d^\beta]}\breve{u}_{i,1}u_{i,1}|}{\lambda_1^{-\frac{1}{2}}
\sqrt{\sum_{i=1}^{d}(\breve{u}_{i,1})^2 }}.\label{innerproduct}
\end{eqnarray}
Below we need to bound the denominator and the numerator of~\eqref{innerproduct}.\\

We start with the numerator. Since $\tilde{X}_i=u_{i,1}\tilde{Z}_1+\tilde{H}_i, i=1,\cdot\cdot\cdot,d$,
it follows that $\tilde{v}^T_1\tilde{X}_i=u_{i,1}\tilde{v}^T_1\tilde{Z}_1+\tilde{v}^T_1\tilde{H}_i
$, which yields
\begin{eqnarray*}
\breve{u}_{i,1}&=&u_{i,1}\tilde{v}^T_1\tilde{Z}_11_{\{|\tilde{v}^T_1\tilde{X}_i|>\lambda\}}
+\tilde{v}^T_1\tilde{H}_i1_{\{|\tilde{v}^T_1\tilde{X}_i|>\lambda\}}\\
&=&u_{i,1}\tilde{v}^T_1\tilde{Z}_1+u_{i,1}\tilde{v}^T_1\tilde{Z}_11_{\{|\tilde{v}^T_1\tilde{X}_i|\leq\lambda\}}
+\tilde{v}^T_1\tilde{H}_i1_{\{|\tilde{v}^T_1\tilde{X}_i|>\lambda\}},
\end{eqnarray*}
and
\begin{eqnarray*}
\sum_{i=1}^{[d^\beta]}\breve{u}_{i,1}u_{i,1}&=&\sum_{i=1}^{[d^\beta]}u^2_{i,1}\tilde{v}^T_1\tilde{Z}_11_{\{|\tilde{v}^T_1\tilde{X}_i|>
\lambda\}}+\sum_{i=1}^{[d^\beta]}u_{i,1}\tilde{v}^T_1\tilde{H}_i1_{\{|\tilde{v}^T_1\tilde{X}_i|>\lambda\}}\\
&=&\tilde{v}^T_1\tilde{Z}_1+\sum_{i=1}^{[d^\beta]}u^2_{i,1}\tilde{v}^T_1\tilde{Z}_11_{\{|\tilde{v}^T_1\tilde{X}_i|\leq\lambda\}}
+\sum_{i=1}^{[d^\beta]}u_{i,1}\tilde{v}^T_1\tilde{H}_i1_{\{|\tilde{v}^T_1\tilde{X}_i|>\lambda\}}.
\end{eqnarray*}
It follows that
\begin{eqnarray}
 \label{upestimate:01}
 \lambda_1^{-\frac{1}{2}}|\sum_{i=1}^{[d^\beta]}\breve{u}_{i,1}u_{i,1}|
 &\leq&\lambda_1^{-\frac{1}{2}}\sum_{i=1}^{[d^\beta]}u^2_{i,1}|\tilde{v}^T_1\tilde{Z}_1|
+
\lambda_1^{-\frac{1}{2}}\sum_{i=1}^{[d^\beta]}|u_{i,1}
\tilde{v}^T_1\tilde{H}_i| \\ \nonumber
&=&|\tilde{v}^T_1\tilde{W}_1|+\sum_{i=1}^{[d^\beta]}\sum_{j=1}^{n}\lambda_1^{-\frac{1}{2}}|u_{i,1}h_{i,j}|,
\end{eqnarray}
and
\begin{eqnarray}
 \label{upestimate:02}
 && \lambda_1^{-\frac{1}{2}}|\sum_{i=1}^{[d^\beta]}\breve{u}_{i,1}u_{i,1}|
\\ \nonumber && \geq\lambda_1^{-\frac{1}{2}}|\tilde{v}^T_1\tilde{Z}_1|-\lambda_1^{-\frac{1}{2}}
\sum_{i=1}^{[d^\beta]}u^2_{i,1}|\tilde{v}^T_1\tilde{Z}_1|1_{\{|\tilde{v}^T_1\tilde{X}_i|\leq\lambda\}}
-\lambda_1^{-\frac{1}{2}}\sum_{i=1}^{[d^\beta]}|u_{i,1}\tilde{v}^T_1\tilde{H}_i| \\ \nonumber
&&\geq|\tilde{v}^T_1\tilde{W}_1|-|\tilde{v}^T_1\tilde{W}_1|
\sum_{i=1}^{[d^\beta]}u^2_{i,1}1_{\{|\tilde{v}^T_1\tilde{X}_i|\leq\lambda\}}
-\sum_{i=1}^{[d^\beta]}\sum_{j=1}^{n}\lambda_1^{-\frac{1}{2}}|u_{i,1}h_{i,j}|.
\end{eqnarray}

Next we will show that
\begin{eqnarray}
\sum_{i=1}^{[d^\beta]}\sum_{j=1}^{n}\lambda_1^{-\frac{1}{2}}|u_{i,1}h_{i,j}|
=o_p(d^{-\frac{\varsigma}{2}}), \mbox{where} \; \varsigma \in [0, \alpha-\eta-\theta).
\label{noiseestimate:01}
\end{eqnarray} Since
$H_j=(h_{1,j},\cdot\cdot\cdot,h_{d,j})^T=\sum_{k=2}^d z_{k,j}u_k,
j=1,\cdot\cdot\cdot,n $, it follows that $h_{i,j}=\sum_{k=2}^d
u_{i,k}z_{k,j}=\sum_{k=2}^d
u_{i,k}\lambda_k^{\frac{1}{2}}w_{k,j} \sim
\mbox{N}(0,\sigma^2_{i,j})$, where $\sigma^2_{i,j}\leq
\lambda_2$, $i=1,\cdot\cdot\cdot, [d^\beta]$,
$j=1,\cdot\cdot\cdot,n$. Thus, for fix $\tau$
\begin{eqnarray*}
&&P\left[\sum_{i=1}^{[d^\beta]}\sum_{j=1}^{n}d^{\frac{\varsigma}{2}}\lambda_1^{-\frac{1}{2}}|u_{i,1}h_{i,j}|\geq\tau\right]\leq
P\left[\bigcup_{i=1}^{[d^\beta]}\left\{\sum_{j=1}^{n}|u_{i,1}h_{i,j}|\geq d^{-\frac{\varsigma}{2}}\lambda_1^{\frac{1}{2}}\tau
u^2_{i,1}\right\}\right]\\
&&\leq\sum_{i=1}^{[d^\beta]}\sum_{j=1}^{n}
P\left[|h_{i,j}|\geq n^{-1} d^{-\frac{\varsigma}{2}} \lambda_1^{\frac{1}{2}}\tau
|u_{i,1}|\right]\leq\sum_{i=1}^{[d^\beta]}\sum_{j=1}^{n}
P\left[|h_{i,j}\sigma^{-1}_{i,j}|\geq c^*
d^{\frac{\alpha-\eta-\theta-\varsigma}{2}}\right]\\
&&=2n[d^\beta]\int_{c d^{(\alpha-\eta-\theta-\varsigma)/2}}^{+\infty}\frac{1}{\sqrt{2\pi}}
\mbox{exp}\left\{-\frac{x^2}{2}\right\}dx\longrightarrow 0, \mbox{as} \;d\rightarrow\infty,
\end{eqnarray*}
where $c$ is constant. Similar, we can show that
 \begin{eqnarray}
 \sum_{i=1}^{[d^\beta]}\lambda_1^{-1}(\sum_{j=1}^{n}|h_{i,j}|)^2=o_p(d^{\frac{\varsigma}{2}})
 ,\label{noiseestimate:02}
 \end{eqnarray}
and
\begin{eqnarray}
\sum_{i=[d^{\beta}]+1}^{d}\sum_{j=1}^{n} \lambda_1^{-\frac{1}{2}}|h_{i,j}|1_{\{\sum_{j=1}^{n}|h_{i,j}|>\lambda
\}}=o_p(d^{-\frac{\varsigma}{2}}),\label{noiseestimate:03}
\end{eqnarray}
 where $\varsigma \in [0, \alpha-\eta-\theta)$.

Finally, we want to show that
\begin{eqnarray}
\sum_{i=1}^{[d^\beta]}u^2_{i,1}
1_{\{|\tilde{v}^T_1\tilde{X}_i|\leq\lambda\}}=o_p(d^{-\frac{\varsigma^{'}}{2}})
,\label{noiseestimate:04}
\end{eqnarray}
where $\varsigma^{'}$ satisfies that $d^{\frac{\varsigma^{'}+\eta-\alpha}{2}}\lambda =o(1)$.
Since we can always find a subsequence of $\{\frac{\lambda_1}{\sum_{i=2}^d \lambda_i}\}$ and make it convergent to a nonnegative constant,
for simplicity, we just assume that  $\mbox{lim}_{d\rightarrow\infty}\frac{\lambda_1}{\sum_{i=2}^d \lambda_i}=C$. If $C=0$, then
the spike index $\alpha<1$, and Jung and Marron (2009)~\cite{jung2009pca}
showed that
\begin{eqnarray*}
 c_d^{-1} S_d &\xrightarrow{p}& I_n,  \;\mbox{as} \; d\rightarrow\infty,
\end{eqnarray*}
where $c_d=n^{-1}\sum_{i=1}^d \lambda_i$. Since the eigenvector
$\tilde{v}^T_1$ of $c_d^{-1} S_d$ can be chosen so that they are
continuous according to Acker (1974)~\cite{acker1974absolute}, it follows that
$\tilde{v}^T_1\Rightarrow v_1$, as $d\rightarrow \infty$, where
$\Rightarrow$ denotes the convergence in distribution and $v_1$ is
the first eigenvector of $n$-dimensional identity matrix. If $C=0$,
then the spike index $\alpha=1$ and Jung, Sen and Marron (2010)~\cite{Jung2010} showed that
$\tilde{v}^T_1\Rightarrow \frac{\tilde{W}_1}{\|\tilde{W}_1\|}$, as
$d\rightarrow \infty$. Therefore, we have
\begin{eqnarray}
\mid\tilde{v}^T_1\tilde{W}_1 \mid \Rightarrow \mid v^T_1\tilde{W}_1
\mid\mbox{\rm or}\; \| \tilde{W}_1\|, \mbox{as}\; d\rightarrow\infty.
\label{v:convergence}
\end{eqnarray}
Since $d^{\frac{\varsigma^{'}+\eta-\alpha}{2}}\lambda =o(1)$,  $
d^{\frac{\varsigma^{'}+\eta-\alpha}{2}}\sum_{j=1}^{n}\mbox{max}_{1\leq i\leq
[d^\beta]}|h_{i,j}|=o_p(1)
$,
and
\begin{eqnarray*}
&&\sum_{i=1}^{[d^\beta]} d^{\frac{\varsigma^{'}}{2}}u^2_{i,1}
1_{\left\{|\tilde{v}^T_1\tilde{X}_i|\leq\lambda\right\}}\leq \sum_{i=1}^{[d^\beta]} d^{\frac{\varsigma^{'}}{2}} u^2_{i,1}
1_{\left\{|u_{i,1}\tilde{v}^T_1\tilde{Z}_1|\leq|\tilde{v}^T_1\tilde{H}_i|+\lambda\right\}}
\\ &&\leq
 \sum_{i=1}^{[d^\beta]} d^{\frac{\varsigma^{'}}{2}} u^2_{i,1}
1_{\left\{|\tilde{v}^T_1\tilde{W}_1|\leq
\lambda_1^{-\frac{1}{2}}\mbox{max}_{1\leq i\leq
[d^\beta]}{|u_{i,1}|}^{-1}\left(\sum_{j=1}^{n}\mbox{max}_{1\leq i\leq
[d^\beta]}|h_{i,j}|+\lambda\right)\right\}}\\
&&\leq\frac{c d^{\frac{\varsigma^{'}+\eta-\alpha}{2}}\sum_{j=1}^{n}\mbox{max}_{1\leq i\leq [d^\beta]}|h_{i,j}|
+c d^{\frac{\varsigma^{'}+\eta-\alpha}{2}}
\lambda}{|\tilde{v}^T_1\tilde{W}_1|},
\end{eqnarray*}
where $c$ is a constant, it follows that  (\ref{noiseestimate:04}) is established.

Then, from (\ref{upestimate:01}), (\ref{upestimate:02}), (\ref{noiseestimate:01}), and (\ref{noiseestimate:04}), we obtain the following result about the numerator
\begin{eqnarray}
\lambda_1^{-\frac{1}{2}}|\sum_{i=1}^{[d^\beta]}\breve{u}_{i,1}u_{i,1}|
=|\tilde{v}^T_1 \tilde{W}_1|+o_p(d^{-\frac{\mbox{min}\{\varsigma,\varsigma^{'}\}}{2}}).\label{converge:up}
\end{eqnarray}

Similarly for the denominator, we have
\begin{eqnarray}
\label{lowestimate:01}
&&\lambda_1^{-\frac{1}{2}}\sqrt{\sum_{i=1}^{d}\breve{u}_{i,1}^2} \leq \lambda_1^{-\frac{1}{2}}\sqrt{\sum_{i=1}^{[d^\beta]}\breve{u}_{i,1}^2}
+\lambda_1^{-\frac{1}{2}}\sqrt{\sum_{i=[d^\beta]+1}^{d}\breve{u}_{i,1}^2}
 \\ \nonumber
&&\leq\lambda_1^{-\frac{1}{2}}\sqrt{\sum_{i=1}^{[d^\beta]}\left(u_{i,1}\tilde{v}^T_1\tilde{Z}_1\right)^2
}+\lambda_1^{-\frac{1}{2}}\sqrt{\sum_{i=1}^{[d^\beta]}(\tilde{v}^T_1\tilde{H}_i)^2}+\lambda_1^{-\frac{1}{2}}
\sum_{i=[d^\beta]+1}^{d}|\breve{u}_{i,1}|
\\ \nonumber && =|\tilde{v}^T_1\tilde{W}_1|+\sqrt{\sum_{i=1}^{[d^\beta]}\lambda_1^{-1}(\sum_{j=1}^{n}|h_{i,j}|)^2}
+\\ \nonumber
&&+
\sum_{i=[d^{\beta}]+1}^{d}\sum_{j=1}^{n}\lambda_1^{-\frac{1}{2}}|h_{i,j}|1_{\{\sum_{j=1}^{n}|h_{i,j}|>\lambda\}},
\end{eqnarray}
and
\begin{eqnarray}
\label{lowestimate:02}
 &&\lambda_1^{-\frac{1}{2}}\sqrt{\sum_{i=1}^{d}\breve{u}_{i,1}^2}\geq
\lambda_1^{-\frac{1}{2}}\sqrt{\sum_{i=1}^{[d^\beta]}\breve{u}_{i,1}^2} \\ \nonumber
&&
\geq|\tilde{v}^T_1\tilde{W}_1|-|\tilde{v}^T_1\tilde{W}_1|\sqrt{\sum_{i=1}^{[d^\beta]}u^2_{i,1}1_{\{|\tilde{v}^T_1\tilde{X}_i|\leq
\lambda\}}}
-\sqrt{\sum_{i=1}^{[d^\beta]}\lambda_1^{-1}(\sum_{j=1}^{n}|h_{i,j}|)^2}.
\end{eqnarray}

Combining (\ref{lowestimate:01}), (\ref{lowestimate:02}), (\ref{noiseestimate:02}), (\ref{noiseestimate:03})  and
(\ref{noiseestimate:04}), we have
\begin{eqnarray}
\lambda_1^{-\frac{1}{2}}\sqrt{\sum_{i=1}^{d}\breve{u}_{i,1}^2}
=|\tilde{v}^T_1 \tilde{W}_1|+o_p(d^{-\frac{\mbox{min}\{\varsigma,\varsigma^{'}\}}{2}}).\label{converge:down}
\end{eqnarray}

Furthermore, (\ref{innerproduct}), (\ref{v:convergence}), (\ref{converge:up}), and (\ref{converge:down}) suggest that
\begin{eqnarray*}
|<\hat{u}^{\rm ST}_1, u_1>|=\frac{|\tilde{v}^T_1 \tilde{W}_1|+o_p(d^{-\frac{\mbox{min}\{\varsigma,\varsigma^{'}\}}{2}})}
{|\tilde{v}^T_1 \tilde{W}_1|+o_p(d^{-\frac{\mbox{min}\{\varsigma,\varsigma^{'}\}}{2}})}
=1 +o_p(d^{-\frac{\mbox{min}\{\varsigma,\varsigma^{'}\}}{2}}),
\label{}
\end{eqnarray*}
which means
that $\hat{u}^{\rm ST}_1$ is consistent with $u_1$ with convergence rate $d^{-\frac{\mbox{min}\{\varsigma,\varsigma^{'}\}}{2}}$. This concludes the proof for Theorem~\ref{Th:02}.\\

In addition, note that $d^{\frac{\varsigma^{'}+\eta-\alpha}{2}}\lambda =o(1)$. If
$\lambda=o(d^{\frac{\alpha-\eta-\varsigma}{2}})$, then we can take $\varsigma^{'}=\varsigma$.
Then $\hat{u}^{\rm ST}_1$ is consistent with $u_1$ with convergence rate $d^{\frac{\varsigma}{2}}$.
This finishes the proof of Theorem~\ref{corr:01}.

\subsection{Proofs of Theorem~\ref{Th:03}, \ref{corr:02},
\ref{Th:pc}, \ref{Th:04} and \ref{corr:03}}\label{subsec:72}
The proofs of Theorems~\ref{Th:03}, \ref{corr:02},
 \ref{Th:04} and \ref{corr:03} are modifications of the
proofs of Theorems~\ref{Th:02} and~\ref{corr:01}. These are provided in the supplementary material, available online at~\cite{shen2011proof}.
The proof of Theorem~\ref{Th:pc} is also given in the supplement.

\subsection{Proof of Theorem~\ref{inconsistent}}\label{subsubsec:73}
Since $X_j^*=(I_{\lfloor d^\beta \rfloor}, (0)_{\lfloor d^\beta \rfloor \times (d-\lfloor d^\beta \rfloor)})X_j$, where
$I_{\lfloor d^\beta \rfloor}$ denotes the $\lfloor d^\beta \rfloor$-dimensional identity matrix and
$(0)_{\lfloor d^\beta \rfloor \times (d-\lfloor d^\beta\rfloor)}$ is the
$\lfloor d^\beta \rfloor$-by-$(d-\lfloor d^\beta\rfloor)$
zero matrix, $j=1,\ldots,n$, it follows that
\begin{eqnarray*}
\Sigma^*_{\lfloor d^\beta\rfloor}=(I_{\lfloor d^\beta \rfloor}, (0)_{\lfloor d^\beta \rfloor \times (d-\lfloor d^\beta \rfloor)})\Sigma_d
(I_{\lfloor d^\beta \rfloor}, (0)_{\lfloor d^\beta \rfloor \times (d-\lfloor d^\beta \rfloor)})^T,
 \end{eqnarray*}
which yields
\begin{eqnarray*}
\lambda^*_1=\lambda_1,  \lambda_2\geq\lambda^*_i\geq \lambda_d, j=2,\ldots,\lfloor d^\beta\rfloor.
\end{eqnarray*}
Therefore,
\begin{eqnarray}
\frac{\sum_{i=2}^{\lfloor d^\beta\rfloor}{\lambda^*_i}^2}{(\sum_{i=2}^{\lfloor d^\beta\rfloor}\lambda^*_i)^2}
\leq \frac{\lfloor d^\beta\rfloor\lambda^2_2}{(\lfloor d^\beta\rfloor)^2\lambda^2_d}=\frac{O(d^\beta)O(d^{2\theta})}{O(d^{2\beta})}
=o(1),\label{Jungcondition:2}
\end{eqnarray}
and
\begin{eqnarray}
\frac{\lambda^*_1}{\sum_{i=2}^{\lfloor d^\beta\rfloor}\lambda^*_i}\leq \frac{\lambda_1}{\lfloor d^\beta\rfloor \lambda_d}=\frac{O(d^\alpha)}{O(d^\beta)}=o(1).
\label{Jungcondition:1}
\end{eqnarray}
If we rescale $\lambda^*_i$, $i=1,\ldots,\lfloor d^\beta\rfloor$,
(\ref{Jungcondition:2}) satisfies the $\varepsilon_2$ assumption of
Jung and Marron (2009)~\cite{jung2009pca} and (\ref{Jungcondition:1}) satisfies the
assumption $\lambda_1=O(d^{\alpha})$ and
$\sum_{i=2}^{d}\lambda_i=O(d)$, where $\alpha<1$. For this case,
Jung and Marron (2009)~\cite{jung2009pca} have shown that $\hat{u}^*_1$ is strongly
inconsistent with $u^*_1$. This means that the oracle estimator
$\hat{u}^{\rm OR}_1$ is strongly inconsistent with $u_1$.

\section*{Acknowledgements}
See \ref{suppA} and~\ref{suppB} for the supplementary materials.

\begin{supplement}
\sname{Supplement A}\label{suppA}
\stitle{Proofs of Theorems 3.1, 3.2, 3.3, 3.4 and 3.5}
\slink[url]{http://www.unc.edu/~dshen/RSPCA/A.pdf}
\sdescription{Detailed technical proofs are provided for Theorems 3.1-3.5.}

\sname{Supplement B}\label{suppB}
\stitle{Additional Simulation Results}
\slink[url]{http://www.unc.edu/~dshen/RSPCA/B.html}
\sdescription{Additional simulation results are provided for the twenty spike index and sparsity index pairs, indicated in Figure~\ref{fig:01}.}
\end{supplement}

\bibliographystyle{imsart-number}
\bibliography{SPCA-HDLSS}

\end{document}